\documentclass{amsart}
\title[Frobenius-Schur theorem]{Frobenius-Schur theorem for $C^*$-categories}
\author{Kenichi Shimizu}
\usepackage{amsmath,amssymb,amscd,amsthm}

\numberwithin{equation}{section}

\newtheorem{counter}            {}[section]
\theoremstyle{definition}
\newtheorem{definition}         [counter]{Definition}
\newtheorem{problem}            [counter]{Problem}
\theoremstyle{plain}
\newtheorem{lemma}              [counter]{Lemma}
\newtheorem{proposition}        [counter]{Proposition}
\newtheorem{theorem}            [counter]{Theorem}

\theoremstyle{remark}
\newtheorem{remark}             [counter]{Remark}

\DeclareMathOperator{\id}{id}

\DeclareMathOperator{\Trace}{Tr}
\newcommand{\Irr}{\mathrm{Irr}}
\newcommand{\Obj}{\mathrm{Obj}}
\newcommand{\Hom}{\mathrm{Hom}}
\newcommand{\End}{\mathrm{End}}
\newcommand{\Rep}{\mathrm{Rep}}
\newcommand{\fd}{\mathit{fd}}
\newcommand{\fdRep}{\mathrm{Rep}_{\fd}}

\newcommand{\fdCorep}{\mathrm{Corep}_{\fd}}

\newcommand{\Mod}  [1]{{{}_{#1}\mathcal{M}}}
\newcommand{\fdMod}[1]{\Mod{#1}_\mathit{fd}}

\begin{document}

\begin{abstract}
  We generalize the Frobenius-Schur theorem to $C^*$-categories. From this category-theoretical point of view, we introduce the notions of real, complex and quaternionic representations of Hopf $C^*$-algebras. Based on these definitions, we give another type of the Hopf-algebraic analogue of the Frobenius-Schur theorem, originally due to Linchenko and Montgomery. Also given are similar results for weak Hopf $C^*$-algebras, table algebras and compact quantum groups.
\end{abstract}

\maketitle

\section{Introduction}
\label{sec:introduction}

\subsection{Frobenius-Schur theorem for compact groups}

We work over the field $\mathbb{C}$ of complex numbers. Let $G$ be a compact group, and let $V$ be a finite-dimensional irreducible continuous representation of $G$ with character $\chi_V$. We say that $V$ is {\em real} if it admits a basis $\{ v_i \}_{i = 1}^n$ such that the matrix representation $\rho: G \to GL_n(\mathbb{C})$ with respect to $\{ v_i \}$ has the following property:
\begin{equation*}
  \rho(G) \subset GL_n(\mathbb{R}).
\end{equation*}
Note that if $V$ is real, then $\chi_V(G) \subset \mathbb{R}$. Following, $V$ is said to be {\em pseudo-real}, or {\em quaternionic}, if it does not admit such a basis but $\chi_V(G) \subset \mathbb{R}$. Finally, $V$ is said to be {\em complex} if it is neither real nor quaternionic. The {\em Frobenius-Schur theorem} gives a way to determine whether $V$ is real, complex or quaternionic: Define the {\em Frobenius-Schur indicator} of $V$ by
\begin{equation}
  \label{eq:intro-FS-ind}
  \nu(V) = \int_G \chi_V(g^2) d \mu(g),
\end{equation}
where $\mu$ is the normalized Haar measure on $G$. Then the theorem states:
\begin{equation}
  \label{eq:intro-FS-ind-RCH}
  \nu(V) =
  \begin{cases}
    + 1 & \text{\rm if $V$ is real}, \\
    \ 0 & \text{\rm if $V$ is complex}, \\
    - 1 & \text{\rm if $V$ is quaternionic}.
  \end{cases}
\end{equation}  

\subsection{Generalizations of the Frobenius-Schur indicator}

In this paper, we give a generalization of the Frobenius-Schur theorem to $C^*$-categories. We should note that the Frobenius-Schur indicator has been generalized in various contexts \cite{MR1657800,MR2029790,MR1078503,MR1808131,MR2104908,MR2313527,MR2381536,MR2366965,MR2725181,MR2879228,MR2095575,KenichiShimizu:2012-11}. Before we describe a summary of our results, we briefly review some of other existing generalizations and raise some questions.

We first mention the result of Linchenko and Montgomery \cite{MR1808131}. For an irreducible representation $V$ of a finite-dimensional semisimple Hopf algebra $H$ (over $\mathbb{C}$), they defined the Frobenius-Schur indicator of $V$ by
\begin{equation}
  \label{eq:intro-FS-ind-Hopf}
  \nu(V) = \chi_V(\Lambda_{(1)} \Lambda_{(2)}),
\end{equation}
where $\chi_V$ is the character of $V$ and $\Lambda_{(1)} \otimes \Lambda_{(2)} \in H \otimes H$ is the comultiplication of the Haar integral $\Lambda \in H$ (in the Sweedler notation). The main result of \cite{MR1808131} is then described as follows:
\begin{enumerate}
\item $\nu(V)$ is one of $+1$, $0$ or $-1$. Moreover, $\nu(V) \ne 0$ if and only if there exists a non-degenerate $H$-invariant bilinear form $\beta: V \times V \to \mathbb{C}$.
\item Suppose $\nu(V) \ne 0$. If $\beta$ is a non-degenerate $H$-invariant bilinear form on $V$, then $\beta(w, v) = \nu(V) \beta(v, w)$ for all $v, w \in V$. In other words, $\beta$ is symmetric if $\nu(V) = +1$ and is skew-symmetric if $\nu(V) = -1$.
\end{enumerate}
One may wonder if this result is a generalization of~\eqref{eq:intro-FS-ind-RCH}. In the case where $H$ is the group algebra, it is not so difficult to derive \eqref{eq:intro-FS-ind-RCH} from the above result. In general, the notion a real, a complex and a quaternionic representation of $H$ is not clear; thus we should rather mention the following problem:

\begin{problem}
  \label{problem:1}
  When can we define a real, a complex and a quaternionic representation of a Hopf algebra? If appropriate definitions are given, how they relate to the Frobenius-Schur indicator?
\end{problem}

Putting this problem aside for a moment, we mention some category-theoretical generalizations of the Frobenius-Schur indicator \cite{MR1657800,MR2029790,MR2313527,MR2381536,MR2366965,MR2725181,KenichiShimizu:2012-11}. In \cite{KenichiShimizu:2012-11}, the author introduced the Frobenius-Schur indicator for categories with duality to unify various generalizations of the Frobenius-Schur theorem. However, the results of \cite{KenichiShimizu:2012-11} describe how the Frobenius-Schur indicator relates to the existence of a kind of invariant bilinear forms and, in particular, are not formulated like \eqref{eq:intro-FS-ind-RCH} (see \cite[Remark~1.4]{KenichiShimizu:2012-11}). The next problem is:

\begin{problem}
  \label{problem:2}
  What is a category-theoretical counterpart of~\eqref{eq:intro-FS-ind-RCH}?
\end{problem}

Namely, we would like to introduce appropriate definitions of `real', `complex' and `quaternionic' objects in some category-theoretical way and then show that the Frobenius-Schur indicator in the sense of \cite{KenichiShimizu:2012-11} detects whether a simple object is real, complex or quaternionic.

\subsection{Summary of results}

Keeping these problems in mind, we now give a summary of our results. Let $\mathcal{A}$ be a $\mathbb{C}$-linear abelian category. Moreover, we assume that $\mathcal{A}$ is {\em locally finite-dimensional} in the following sense:
\begin{equation*}
  \dim_{\mathbb{C}} \Hom_{\mathcal{A}}(X,Y) < \infty \quad (X, Y \in \mathcal{A}).
\end{equation*}

In \S\ref{sec:dual-structures}, we introduce the notion of a {\em dual structure} for $\mathcal{A}$ and an equivalence relation between them. For each dual structure $\mathcal{D}$ for $\mathcal{A}$, we define a function
\begin{equation*}
  \nu_{\mathcal{D}}: \Obj(\mathcal{A}) \to \mathbb{Z},
\end{equation*}
where $\Obj(\mathcal{A})$ is the isomorphism classes of objects of $\mathcal{A}$. Following \cite{KenichiShimizu:2012-11}, $\nu_{\mathcal{D}}$ is called the {\em Frobenius-Schur indicator} with respect to $\mathcal{D}$. We show that the function $\nu_{\mathcal{D}}$ depends only on the equivalence class of $\mathcal{D}$.

Now let $A$ be an algebra over $\mathbb{C}$. In our applications, important are dual structures for the category $\fdMod{A}$ of finite-dimensional left $A$-modules. By a {\em dual structure} for $A$, we mean a pair $(S, g)$ consisting of an anti-algebra map $S: A \to A$ and an element $g \in A$ satisfying certain conditions. Such a pair $(S, g)$ yields a dual structure for $\fdMod{A}$. In \S\ref{subsec:DS-alg}, we recall from  \cite{KenichiShimizu:2012-11} a formula of the Frobenius-Schur indicator with respect to this type of dual structures.

In \S\ref{sec:real-structures}, we introduce the notion of a {\em real structure} for $\mathcal{A}$ and an equivalence relation between them. Given a real structure $\mathcal{J}$ for $\mathcal{A}$, the {\em $\mathcal{J}$-signature} $\sigma_{\mathcal{J}}(X) \in \{ 0, \pm 1 \}$ is defined for each simple object $X \in \mathcal{A}$ according to whether there exists a certain type of isomorphisms. Taking the $\mathcal{J}$-signature is a function
\begin{equation*}
  \sigma_{\mathcal{J}}: \Irr(\mathcal{A}) \to \{ 0, \pm 1 \},
\end{equation*}
where $\Irr(\mathcal{A})$ is the class of isomorphism classes of simple objects of $\mathcal{A}$. We show that the function $\sigma_{\mathcal{J}}$ depends only on the equivalence class of $\mathcal{J}$.

A {\em real form} of an algebra $A$ is an $\mathbb{R}$-subalgebra $A_0 \subset A$ such that $A = A_0 \oplus \mathbf{i} A_0$, where $\mathbf{i} = \sqrt{-1}$. In \S\ref{subsec:RF-alg}, we first observe that a real form of $A$ yields a real structure for $\fdMod{A}$. We say that a simple module $V \in \fdMod{A}$ is said to be {\em real}, {\em complex} and {\em quaternionic} with respect to $A_0$ if $\sigma_{\mathcal{J}}(V)$ is equal to $+1$, $0$ and $-1$, respectively, where $\mathcal{J}$ is the real structure for $\fdMod{A}$ obtained from $A_0$. These notions can be characterized in more familiar ways. For example, $V$ is real if and only if it admits a basis such that the matrix representation $\rho: A \to M_n(\mathbb{C})$ with respect to that basis satisfies $\rho(A_0) \subset M_n(\mathbb{R})$; see Theorems~\ref{thm:RF-sign-1} and~\ref{thm:RF-sign-2} for details.

In \S\ref{sec:FS-thm-C-star-cat}, we assume moreover that $\mathcal{A}$ is a $C^*$-category. Let $\mathsf{DS}_*(\mathcal{A})$
({\it resp}. $\mathsf{RS}_*(\mathcal{A})$) be the equivalence classes of dual ({\it resp}. real) structures for $\mathcal{A}$ which are compatible with the $*$-structure ({\em $*$-compatibility}, see Definition \ref{def:star-compatible}). There exists a bijection
\begin{equation}
  \label{eq:intro-real-dual-bij}
  \mathbb{D}: \mathsf{RS}_*(\mathcal{A}) \to \mathsf{DS}_*(\mathcal{A}),
\end{equation}
see Theorem~\ref{thm:real-vs-dual}. Let $\mathcal{J}$ be a $*$-compatible real structure for $\mathcal{A}$. The existence of the bijection suggests that the $\mathcal{J}$-signature can be computed by using the corresponding dual structure for $\mathcal{A}$ in some way. We show that the way is the Frobenius-Schur indicator: If the equivalence class of a $*$-compatible dual structure $\mathcal{D}$ for $\mathcal{A}$ corresponds to the equivalence class of $\mathcal{J}$ via~\eqref{eq:intro-real-dual-bij}, then
\begin{equation}
  \label{eq:intro-FS-thm-C-st}
  \nu_{\mathcal{D}}(X) = \sigma_{\mathcal{J}}(X)
\end{equation}
for all simple object $X \in \mathcal{A}$ (Theorem~\ref{thm:FS-thm-C-st}). This answers to Problem~\ref{problem:2}; in \S\ref{sec:FS-thm-fd-C-star}--\S\ref{sec:CQG}, we explain why \eqref{eq:intro-FS-thm-C-st} can be considered as a generalization of the Frobenius-Schur theorem.

In \S\ref{sec:FS-thm-fd-C-star}, we apply \eqref{eq:intro-FS-thm-C-st} to finite-dimensional $C^*$-algebras. Let $A$ be such an algebra, and let $S: A \to A$ be an anti-algebra map such that
\begin{equation}
  \label{eq:S-st-S-st=id}
  S(S(a)^*)^* = a \quad (a \in A).
\end{equation}
Then $A_0 = \{ a \in A \mid S(a)^* = a \}$ is a real form of $A$ (and, moreover, any real form of $A$ is obtained in this way). Now we consider the category $\mathcal{A} := \fdRep(A)$ of finite-dimensional $*$-representations of $A$. The map $S$ and the real form $A_0$ define a dual structure $\mathcal{D}$ and a real structure $\mathcal{J}$ for $\mathcal{A}$, respectively. We show that $\mathcal{D}$ and $\mathcal{J}$ are $*$-compatible and their equivalence classes correspond via \eqref{eq:intro-real-dual-bij}. Thus \eqref{eq:intro-FS-thm-C-st} implies, for all simple $V \in \Rep(A)$,
\begin{equation}
  \label{eq:intro-FS-thm-fd-st-alg}
  \nu_{\mathcal{D}}(V) =
  \begin{cases}
    + 1 & \text{if $V$ is real}, \\
    \ 0 & \text{if $V$ is complex}, \\
    - 1 & \text{if $V$ is quaternionic (with respect to $A_0$)}.
  \end{cases}
\end{equation}

Now let $A$ be a finite-dimensional Hopf $C^*$-algebra, and let $S: A \to A$ be the antipode of $A$ (which is known to satisfy \eqref{eq:S-st-S-st=id}). The right-hand side of~\eqref{eq:intro-FS-thm-fd-st-alg} is then equal to \eqref{eq:intro-FS-ind-Hopf}. Hence an answer to Problem~\ref{problem:1} is obtained: The Frobenius-Schur indicator for the Hopf algebra $A$, defined by Linchenko and Montgomery, detects whether a given simple $A$-module $V$ is real, complex or quaternionic with respect to the real form
\begin{equation*}
  A_0 = \{ a \in A \mid \text{$S(a)^* = a$, where $S: A \to A$ is the antipode} \}.
\end{equation*}
If $A = \mathbb{C}G$ is the group algebra of a finite group $G$, then $A_0$ is precisely $\mathbb{R}G$. A similar result for finite-dimensional weak Hopf $C^*$-algebras \cite{MR1726707,MR1793595}, that for table algebras \cite{MR2535395}, and their `twisted' versions are also deduced from~\eqref{eq:intro-FS-thm-fd-st-alg}.

In \S\ref{sec:CQG}, we apply \eqref{eq:intro-FS-thm-C-st} to compact quantum groups \cite{MR1616348,MR1310296,MR1431306}. Since, by definition, a representation of a compact quantum group is a comodule over its so-called quantum coordinate algebra, we first prove the results of \S\ref{subsec:DS-alg}, \S\ref{subsec:RF-alg} and \S\ref{sec:FS-thm-fd-C-star} in coalgebraic settings. Once such results are developed, we easily derive an exact quantum analogue of the Frobenius-Schur theorem for compact groups (Theorem~\ref{thm:FS-thm-CQG}).

In Appendix~A, we mention a proposition due to B\"ohm, Nill and Szlach\'anyi \cite{MR1726707} (see Proposition \ref{prop:BNS}) and related problems. Since their lemma plays a quite important role in \S\ref{sec:FS-thm-fd-C-star}, it is worth to give a new proof from the viewpoint of our theory. We prove Proposition \ref{prop:BNS} by emphasizing the `lifting problem' (Remark~\ref{rem:lifting-problem}) of a certain functor. A coalgebraic version of Proposition \ref{prop:BNS} is also proved.

\section*{Acknowledgments}

The author is supported by Grant-in-Aid for JSPS Fellows (24$\cdot$3606).

\section{Dual structures}
\label{sec:dual-structures}

\subsection{Dual structures}

Following Balmer \cite{MR2181829}, a {\em category with duality} is a triple $(\mathcal{A}, D, \eta)$ consisting of a category $\mathcal{A}$, a contravariant endofunctor $D$ on $\mathcal{A}$, and a natural isomorphism $\eta: \id_{\mathcal{A}} \to D D$ satisfying
\begin{equation}
  \label{eq:dual-str-1}
  D(\eta_X) \circ \eta_{D(X)} = \id_{D(X)} \quad (X \in \mathcal{A}).
\end{equation}
Since we will deal at the same time with many pairs $(D, \eta)$ such that $(\mathcal{A}, D, \eta)$ is a category with duality, it is convenient to introduce the following terminology:

\begin{definition}
  Let $\mathcal{A}$ be a $\mathbb{C}$-linear category. A {\em dual structure} for $\mathcal{A}$ is a pair $(D, \eta)$ of a contravariant $\mathbb{C}$-linear endofunctor $D$ on $\mathcal{A}$ and a natural isomorphism $\eta: \id_{\mathcal{A}} \to D D$ satisfying~\eqref{eq:dual-str-1}. If $\mathcal{D} = (D, \eta)$ and $\mathcal{D}' = (D', \eta')$ are dual structures for $\mathcal{A}$, then a {\em morphism} from $\mathcal{D}$ to $\mathcal{D}'$ is a natural transformation $\xi: D \to D'$ such that the following diagram commutes for all $X \in \mathcal{A}$:
  \begin{equation*}
    \begin{CD}
      X @>{\eta_X}>> D D(X) \\
      @V{\eta'_X}V{}V @V{}V{\xi_{D(X)}}V \\
      D' D' (X) @>>{D'(\xi_X)}> D' D(X)
    \end{CD}
  \end{equation*}
\end{definition}

It is trivial from the definition that $(D, \eta)$ is a dual structure for $\mathcal{A}$ if and only if the triple $(\mathcal{A}, D, \eta)$ is a category with duality such that $D$ is $\mathbb{C}$-linear.

Dual structures for $\mathcal{A}$ form a category. If dual structures $\mathcal{D} = (D, \eta)$ and $\mathcal{D}' = (D', \eta')$ are isomorphic in this category, then we say that they are {\em equivalent} and write $\mathcal{D} \sim \mathcal{D}'$. Note that a natural isomorphism $\xi: D \to D'$ is an isomorphism from $\mathcal{D}$ to $\mathcal{D}'$ if and only if
\begin{equation}
  \label{eq:dual-str-3}
  (\id_\mathcal{A}, \xi): (\mathcal{A}, D, \eta) \to (\mathcal{A}, D', \eta')
\end{equation}
is a strong duality preserving functor in the sense of Calm\`es and Hornbostel \cite{MR2520968}.

\subsection{Frobenius-Schur indicator}

Let $\mathcal{A}$ be a locally finite-dimensional $\mathbb{C}$-linear category, and let $X \in \mathcal{A}$ be an object. Given a dual structure $\mathcal{D} = (D, \eta)$ for $\mathcal{A}$, we define a linear map
\begin{equation*}
  T_{\mathcal{D}|X}: \Hom_{\mathcal{A}}(X, D(X)) \to \Hom_{\mathcal{A}}(X, D(X)),
  \quad f \mapsto \eta_Y \circ D(f).
\end{equation*}

\begin{definition}
  The {\em Frobenius-Schur indicator} of $X \in \mathcal{A}$ with respect to the dual structure $\mathcal{D}$ is defined and denoted by $\nu_D(X) = \Trace(T_{\mathcal{D}|X})$, where $\Trace$ is the trace of a linear operator.
\end{definition}

Let $\Obj(\mathcal{A})$ be the isomorphism classes of objects of $\mathcal{A}$. The value of $\nu_{\mathcal{D}}(X)$ depends on the isomorphism class of $X$. Since the square of $T_{\mathcal{D}|X}$ is the identity, its trace $\nu_{\mathcal{D}}(X)$ is an integer. Hence, taking the Frobenius-Schur indicator with respect $\mathcal{D}$ defines a map
\begin{equation*}
  \nu_{\mathcal{D}}: \Obj(\mathcal{A}) \to \mathbb{Z}, \quad X \mapsto \nu_{\mathcal{D}}(X).
\end{equation*}

Now let $\mathcal{D}' = (D', \eta')$ be another dual structure for $\mathcal{A}$. If $\mathcal{D} \sim \mathcal{D}'$, then, since $(\mathcal{A}, D, \eta)$ and $(\mathcal{A}, D', \eta')$ are equivalent as categories with duality via \eqref{eq:dual-str-3}, we have $\nu_{\mathcal{D}}(X) = \nu_{\mathcal{D}'}(X)$ for all $X \in \mathcal{A}$ \cite[Proposition 2.10]{KenichiShimizu:2012-11}. In other words:

\begin{lemma}
  \label{lem:FS-ind-equiv}
  The function $\nu_{\mathcal{D}}$ depends only on the equivalence class of $\mathcal{D}$.
\end{lemma}

We suppose moreover that $\mathcal{A}$ is an abelian category. Let $X$ be a simple object of $\mathcal{A}$. By Schur's lemma, $\Hom_{\mathcal{A}}(X, D(X))$ is one-dimensional if $X \cong D(X)$ and zero otherwise. Since $(T_{\mathcal{D}|X})^2$ is the identity, we have $\nu_{\mathcal{D}}(X) \in \{ 0, \pm 1 \}$ and
\begin{equation}
  \label{eq:FS-ind-simple-1}
  \nu_{\mathcal{D}}(X) \ne 0 \iff X \cong D(X).
\end{equation}
If $X \cong D(X)$, then $\nu_{\mathcal{D}}(X) \in \{ \pm 1 \}$ is characterized by the following formula:
\begin{equation}
  \label{eq:FS-ind-simple-2}
  D(f) \circ \eta_X = \nu_{\mathcal{D}}(X) \cdot f
  \quad (f \in \Hom_{\mathcal{A}}(X, D(X)),
\end{equation}
see \cite[Proposition~2.12]{KenichiShimizu:2012-11} for details.

\subsection{Dual structures for an algebra}
\label{subsec:DS-alg}

Given an algebra $A$ over $\mathbb{C}$, we denote by $\Mod{A}$ and $\fdMod{A}$ the $\mathbb{C}$-linear abelian category of left $A$-modules and its full subcategory consisting of finite-dimensional objects, respectively.

\begin{definition}
  By a {\em dual structure} for $A$, we mean a pair $(S, g)$ consisting of an anti-algebra map $S: A \to A$ and an invertible element $g \in A$ satisfying
  \begin{equation*}
    S(g) = g^{-1}, \quad S^2(a) = g a g^{-1} \quad (a \in A).
  \end{equation*}  
\end{definition}

A dual structure $(S, g)$ for $A$ gives rise to a dual structure for $\fdMod{A}$. To explain, we introduce some notations: We denote by $X^*$ the set of all linear maps from a vector space $X$ to $\mathbb{C}$. Given a linear map $f: X \to Y$, define
\begin{equation*}
  f^*: Y^* \to X^*, \quad f^*(\psi) = \psi \circ f \quad (\psi \in Y^*).
\end{equation*}
Now, for $X \in \fdMod{A}$, we define $D(X) \in \fdMod{A}$ to be the vector space $X^\ast$ endowed with the left $A$-module structure given by
\begin{equation}
  \label{eq:dual-module-1}
  \langle a \cdot \psi, x \rangle = \langle \psi, S(a) x \rangle \quad (a \in A, \psi \in X^\ast, x \in X).
\end{equation}
For a morphism $f: X \to Y$ in $\fdMod{A}$, set $D(f) = f^*$. $X \mapsto D(X)$ defines a $\mathbb{C}$-linear contravariant endofunctor $D$ on $\fdMod{A}$. Moreover, there is a natural isomorphism $\eta: \id \to D D$ given by
\begin{equation*}
  \langle \eta_X(x), \psi \rangle = \langle \psi, g x \rangle \quad (x \in X, \psi \in X^\ast).
\end{equation*}
The pair $(D, \eta)$ is a dual structure for $\fdMod{A}$, which will be referred to as the {\em dual structure for $\fdMod{A}$ associated with $(S, g)$}.

Recall that a {\em separability idempotent} of $A$ is an element $E = \sum_i E_i' \otimes E_i'' \in A \otimes A$ such that $\sum_i E_i' E_i'' = 1$ and $\sum_i a E_i' \otimes E_i''  = \sum_i E_i' \otimes E_i'' a$ for all $a \in A$. If such an element exists, then:

\begin{theorem}[{\cite[Theorem~3.8]{KenichiShimizu:2012-11}}]
  \label{thm:FS-ind-sep-alg}
  For all $V \in \fdMod{A}$, we have
  \begin{equation}
    \label{eq:FS-formula-sep-alg}
    \nu_{\mathcal{D}}(V) = \sum_i \chi_V(S(E_i') g E_i''),
  \end{equation}
  where $\mathcal{D}$ is the dual structure associated with $(S, g)$.
\end{theorem}

\begin{remark}
  \label{rem:character}
  To express the right-hand side of \eqref{eq:FS-formula-sep-alg} neatly, the following formulas will be used: If $V \in \fdMod{A}$ is simple and $z \in A$ is central, then
  \begin{equation*}
    \chi_V(z a) = \frac{\chi_V(z) \cdot \chi_V(a)}{\chi_V(1)},
    \quad \chi_V(z^{-1} a) = \frac{\chi_V(1) \cdot \chi_V(a)}{\chi_V(z)} \quad (a \in A).
  \end{equation*}
  Here, in the latter formula, $z$ is assumed to be invertible. These formulas follow from that the action of $z$ on $V$ is a scalar multiple by Schur's lemma.
\end{remark}

\section{Real structures}
\label{sec:real-structures}

\subsection{Real structures}

We say that a functor $F: \mathcal{A} \to \mathcal{B}$ between $\mathbb{C}$-linear categories is {\em anti-linear} if the map $F: \Hom_{\mathcal{A}}(X, Y) \to \Hom_{\mathcal{B}}(F(X), F(Y))$ induced by $F$ is an anti-linear map for all $X, Y \in \mathcal{A}$. Now let $\mathcal{A}$ be a $\mathbb{C}$-linear category.

\begin{definition}
  A {\em real structure} for $\mathcal{A}$ is a pair $\mathcal{J} = (J, i)$ consisting of an anti-linear functor $J: \mathcal{A} \to \mathcal{A}$ and a natural isomorphism $i: \id_{\mathcal{A}} \to J J$ satisfying
  \begin{equation*}
    i_{J(X)} = J(i_X)
  \end{equation*}
  for all $X \in \mathcal{A}$. If $\mathcal{J} = (J, i)$ and $\mathcal{J}' = (J', i')$ are real structures for $\mathcal{A}$, then a {\em morphism} from $\mathcal{J}$ to $\mathcal{J}'$ is a natural isomorphism $\beta: J \to J'$ such that the following diagram   commutes for all $X \in \mathcal{A}$:
  \begin{equation}
    \label{eq:real-str-2}
    \begin{CD}
      X @>{i_X}>{}> J J(X) \\
      @V{i_X'}V{}V @V{}V{\beta_{J(X)}}V \\
      J' J'(X) @<{}<{J'(\beta_X)}< J' J(X)
    \end{CD}
  \end{equation}
\end{definition}

Real structures for $\mathcal{A}$ form a category. If $\mathcal{J}$ and $\mathcal{J}'$ are isomorphic in this category, then we say that they are {\em equivalent} and write $\mathcal{J} \sim \mathcal{J}'$.

Beggs and Majid \cite{MR2501177} introduced the notion of {\em bar categories}. In most situations, a bar category can be thought as a $\mathbb{C}$-linear monoidal category endowed with a real structure compatible with the monoidal structure. Thus many examples of real structures are found in \cite{MR2501177}.

\subsection{Signature}

Let $\mathcal{A}$ be a locally finite-dimensional $\mathbb{C}$-linear abelian category, and let $\mathcal{J} = (J, i)$ be a real structure for $\mathcal{A}$. Given a symbol $\varepsilon \in \{ +, - \}$, a {\em $\mathcal{J}_{\varepsilon}$-structure} for $X \in \mathcal{A}$ is an isomorphism $j: X \to J(X)$ satisfying
\begin{equation*}
  J(j) \circ j = \varepsilon i_X: X \to J J(X).
\end{equation*}

\begin{lemma}
  \label{lem:real-str-1}
  Let $\mathcal{J} = (J, i)$ be a real structure for $\mathcal{A}$. If $X \in \mathcal{A}$ is a simple object, then one and only one of the following statements holds:
  \begin{enumerate}
  \item[(1)] $X$ is not isomorphic to $J(X)$.
  \item[(2)] $X$ has a $\mathcal{J}_+$-structure.
  \item[(3)] $X$ has a $\mathcal{J}_-$-structure.
  \end{enumerate}
\end{lemma}
\begin{proof}
  It is trivial that if (1) holds, then both (2) and (3) do not. Now suppose that (1) does not hold. Let $f: X \to J(X)$ be an isomorphism. By Schur's lemma, $J(f) \circ f = \alpha i_X$ for some $\alpha \in \mathbb{C}^\times$. By the definition of a real structure,
  \begin{equation*}
    J J(f) \circ J(f) \circ f
    = J J(f) \circ \alpha i_X
    = \alpha i_{J(X)} \circ f
    = \alpha J(i_X) \circ f.
  \end{equation*}
  On the other hand, since the functor $J$ is anti-linear, we compute
  \begin{equation*}
    J J(f) \circ J(f) \circ f
    = J(J(f) \circ f) \circ f
    = J(\alpha i_X) \circ f
    = \overline{\alpha} J(i_X) \circ f.
  \end{equation*}
  Hence $\alpha = \overline{\alpha}$ follows. Namely, $\alpha \in \mathbb{R}$. Now we set $j = |\alpha|^{-1/2} f$. Then $j$ is a $\mathcal{J}_+$-structure or a $\mathcal{J}_-$-structure for $X$ according to whether $\alpha > 0$ or $\alpha < 0$. Hence, at least either one of (2) or (3) holds.

  To complete the proof, we show that (2) and (3) cannot occur at the same time. Suppose that (2) holds and fix a $\mathcal{J}_+$-structure $j$ for $X$. If $f: X \to J(X)$ is an isomorphism, then, by Schur's lemma, $f = \lambda j$ for some $\lambda \in \mathbb{C}^\times$. Since
  \begin{equation*}
    J(f) \circ f = J(\lambda j) \circ \lambda j = |\lambda|^2 J(j) \circ j = |\lambda|^2 i_X
  \end{equation*}
  and $|\lambda^2| > 0$, $f$ cannot be a $\mathcal{J}_-$-structure. Hence, (3) does not hold. In a similar way, we see that (3) implies the negation of (2).
\end{proof}

\begin{definition}
  Let $X \in \mathcal{A}$ be a simple object. The {\em $\mathcal{J}$-signature} $\sigma_{\mathcal{J}}(X)$ is defined to be $0$, $+1$ or $-1$ according to whether (1), (2) or (3) of Lemma \ref{lem:real-str-1} holds.
\end{definition}

Let $\Irr(\mathcal{A})$ denote the isomorphism classes of simple objects of $\mathcal{A}$. It is easy to see that $\sigma_{\mathcal{J}}(X) = \sigma_{\mathcal{J}}(Y)$ whenever $X \cong Y$. Hence, taking the $\mathcal{J}$-signature can be considered as a map
\begin{equation*}
  \sigma_{\mathcal{J}}: \Irr(\mathcal{A}) \to \mathbb{Z},
  \quad X \mapsto \sigma_{\mathcal{J}}(X).
\end{equation*}

\begin{lemma}
  \label{lem:real-str-2}
  The function $\sigma_{\mathcal{J}}$ depends on the equivalence class of $\mathcal{J}$.
\end{lemma}
\begin{proof}
  Let $\mathcal{J} = (J, i)$ and $\mathcal{J}' = (J', i')$ be real structures for $\mathcal{A}$ and suppose that there exists an isomorphism $\beta: \mathcal{J} \to \mathcal{J}'$ of real structures. Then
  \begin{equation*}
    \sigma_{\mathcal{J}}(X) \ne 0
    \iff X \cong J(X)
    \iff X \cong J'(X)
    \iff \sigma_{\mathcal{J}'}(X) \ne 0
  \end{equation*}
  for all $X \in \Irr(\mathcal{A})$ by the definition of the signature. In particular,
  \begin{equation}
    \label{eq:lem-real-str-2-1}
    \sigma_{\mathcal{J}}(X) = \sigma_{\mathcal{J}'}(X)
  \end{equation}
  holds if $\sigma_{\mathcal{J}}(X) = 0$. Thus we consider the case where $\varepsilon := \sigma_{\mathcal{J}}(X) \ne 0$. Then there exists a $\mathcal{J}_{\varepsilon}$-structure $j: X \to J(X)$ for $X$. By \eqref{eq:real-str-2}, the morphism
  \begin{equation*}
    \begin{CD}
      j': X @>{j}>{}> J(X) @>{\beta_X}>{}> J'(X)
    \end{CD}
  \end{equation*}
  is a $\mathcal{J}'_{\varepsilon}$-structure for $X$. Hence \eqref{eq:lem-real-str-2-1} holds also in this case.
\end{proof}

\subsection{Real forms of an algebra}
\label{subsec:RF-alg}

Let $A$ be an algebra over $\mathbb{C}$. By a {\em real form} of $A$, we mean an $\mathbb{R}$-subalgebra $A_0 \subset A$ such that $A = A_0 \oplus \mathbf{i} A_0$. Suppose that a real form $A_0$ of $A$ is given. For $a \in A$, we define
\begin{equation*}
  \overline{a} = x - \mathbf{i} y
  \quad (x, y \in A_0, a = x + \mathbf{i} y)
\end{equation*}
and call $\overline{a}$ the {\em conjugate} of $a$ with respect to the real form $A_0$. It is easy to see that the map $\overline{\phantom{a}}: a \mapsto \overline{a}$ is an anti-linear operator on $A$ such that
\begin{equation}
  \label{eq:RF-cpx-conj-2}
  \overline{a b} = \overline{a} \cdot \overline{b}, \quad \overline{\overline{a}} = a \quad (a, b \in A).
\end{equation}
Conversely, if an anti-linear operator $\overline{\phantom{a}}: A \to A$ satisfies~\eqref{eq:RF-cpx-conj-2}, then
\begin{equation*}
  A_0 = \{ a \in A \mid \overline{a} = a \}
\end{equation*}
is a real form of $A_0$. Thus giving a real form of $A$ is equivalent to giving an anti-linear operator on $A$ satisfying~\eqref{eq:RF-cpx-conj-2}.

Given a vector space $X$ over $\mathbb{C}$, we denote by $\overline{X}$ its complex conjugate; namely, $\overline{X} = X$ as an abelian group and the action of $\mathbb{C}$ on $\overline{X}$ is determined by
\begin{equation*}
  \overline{c} \cdot \overline{x} = \overline{c x} \quad (c \in \mathbb{C}, x \in X)
\end{equation*}
if we denote by $\overline{x}$ the element $x \in X$ regarded as an element of $\overline{X}$. Now let $A$ be an algebra over $\mathbb{C}$, and let $A_0 \subset A$ be a real form. If $X$ is a left $A$-module, then $\overline{X}$ is also a left $A$-module by the action determined by
\begin{equation*}
  \overline{a} \cdot \overline{x} = \overline{a x} \quad (a \in A, x \in X).
\end{equation*}
Given a morphism $f: X \to Y$ in $\Mod{A}$, we define $\overline{f}: \overline{X} \to \overline{Y}$ in $\Mod{A}$ by
\begin{equation*}
  \overline{f}(\overline{x}) = \overline{f(x)} \quad (x \in X).
\end{equation*}
We call $\overline{X}$ the {\em conjugate} of $X$ with respect to the real form $A_0$. Taking the conjugate of an $A$-module defines an anti-linear functor
\begin{equation}
  \label{eq:RF-conj-rep-3}
  \overline{\phantom{a}}: \Mod{A} \to \Mod{A}, \quad X \mapsto \overline{X}.
\end{equation}
Moreover, there is a natural isomorphism $i: \id \to \overline{\phantom{a}} \circ \overline{\phantom{a}}$ defined by
\begin{equation}
  \label{eq:RF-conj-rep-4}
  i_X: X \to \overline{\overline{X}},
  \quad i_X(x) = \overline{\overline{x}}
  \quad (x \in X \in \Mod{A}).
\end{equation}
The pair $\mathcal{J} = (\overline{\phantom{a}}, i)$ is a real structure for $\Mod{A}$, which will be referred to as the {\em real structure associated with the real form $A_0$}.

We give representation-theoretic interpretations of the $\mathcal{J}$-signature and related notions. Let $V \in \Mod{A}$ and $\varepsilon \in \{ +, - \}$. By the definition of $\overline{V}$, a $\mathcal{J}_{\varepsilon}$-structure for $V$ is nothing but an anti-linear map $j: V \to V$ satisfying
\begin{equation}
  \label{eq:RF-J-eps-str}
  j^2 = \varepsilon \id_V, \quad j(a v) = \overline{a} j(v) \quad (a \in A, v \in V).
\end{equation}
By a {\em real form} of $V$ (with respect to the real form $A_0$), we mean an $A_0$-submodule $V_0 \subset V$ such that $V = V_0 \oplus \mathbf{i} V_0$. The proof of the following lemma is omitted.

\begin{lemma}
  \label{lem:J-plus}
  Given a $\mathcal{J}_+$-structure $j$ for $V$,
  \begin{equation}
    \label{eq:J-plus-1}
    V_0 = \{ v \in V \mid j(v) = v \}
  \end{equation}
  is a real form of $V$. Conversely, given a real form $V_0$ of $V$,
  \begin{equation}
    \label{eq:J-plus-2}
    j: V \to V, \quad j(x + \mathbf{i} y) = x - \mathbf{i} y \quad (x, y \in V_0)
  \end{equation}
  is a $\mathcal{J}_+$-structure for $V$. \eqref{eq:J-plus-1} and \eqref{eq:J-plus-2} establish a bijection between $\mathcal{J}_+$-structures for $V$ and real forms of $V$.
\end{lemma}

The following characterization should be noted:

\begin{lemma}
  \label{lem:J-plus-real-basis}
  $V \in \fdMod{A}$ has a real form if and only if it admits a $\mathbb{C}$-basis $\mathcal{B}$ such that the matrix representation $\rho: A \to M_n(\mathbb{C})$ with respect to $\mathcal{B}$ has the following property: $\rho(A_0) \subset M_n(\mathbb{R})$.
\end{lemma}
\begin{proof}
  If $V$ has a real form $V_0$, then any $\mathbb{R}$-basis $\mathcal{B}$ of $V_0$ is a $\mathbb{C}$-basis of $V$ having such a property. Conversely, if $V$ admits such a basis $\mathcal{B}$, then the $\mathbb{R}$-subspace of $V$ spanned by $\mathcal{B}$ is a real form of $V$.
\end{proof}

Let $\mathbb{H} = \mathbb{R}1 \oplus \mathbb{R} \mathbf{i} \oplus \mathbb{R} \mathbf{j} \oplus \mathbb{R} \mathbf{k}$ be the quaternions, {\it i.e.}, the algebra over $\mathbb{R}$ generated by $\mathbf{i}$, $\mathbf{j}$ and $\mathbf{k}$ with relations $\mathbf{i}^2 = \mathbf{j}^2 = \mathbf{k}^2 = \mathbf{i j k} = -1$. Set $A_{\mathbb{H}} := \mathbb{H} \otimes_{\mathbb{R}} A_0$ and regard $A$ ($\cong \mathbb{C} \otimes_{\mathbb{R}} A_0$) as an $\mathbb{R}$-subalgebra of $A_{\mathbb{H}}$. Since $\mathbf{j i} = - \mathbf{i j}$, we have $\mathbf{j} \cdot a = \overline{a} \cdot \mathbf{j}$ in $A_{\mathbb{H}}$ for all $a \in A$. This observation yields:

\begin{lemma}
  \label{lem:J-minus}
  Let $V \in \Mod{A}$. An $\mathbb{R}$-linear map $j: V \to V$ is a $\mathcal{J}_-$-structure for $V$ if and only if the following formula defines an action of $A_{\mathbb{H}}$:
  \begin{equation*}
    (a + \mathbf{j} b) \cdot v = a + j(b v) \quad (a, b \in A, v \in V)
  \end{equation*}
  Thus $\mathcal{J}_-$-structures for $V$ and $A_{\mathbb{H}}$-module structures on $V$ extending the $A$-module structure on $V$ are in bijection.
\end{lemma}

Given $V \in \fdMod{A}$, we denote by $\chi_V: A \to \mathbb{C}$ the character of $V$. Note that
\begin{equation*}
  N \cong M \iff \chi_N = \chi_M
\end{equation*}
whenever $N, M \in \fdMod{A}$ are simple $A$-modules.

\begin{lemma}
  \label{lem:J-zero}
  Let $V \in \fdMod{A}$. Then:
  \begin{enumerate}
  \item The character of $\overline{V}$ is given by $\chi_{\overline{V}}(a) = \overline{\chi_V(\overline{a})}$ for $a \in A$.
  \item If $V$ is simple, then $V \cong \overline{V}$ is equivalent to $\chi_{V}(A_0) \subset \mathbb{R}$.
  \end{enumerate}
\end{lemma}
\begin{proof}
  (1) Fix a basis $\{ v_i \}_{i = 1}^n$ of $V$. Note that $\{ \overline{v_i} \}_{i = 1}^n$ is a basis of $\overline{V}$. Now let $a \in A$ and suppose that the action of $\overline{a}$ on $V$ is expressed as
  \begin{equation*}
    \overline{a} \cdot v_j = c_{1 j} v_1 + \dotsb + c_{n j} v_n
    \quad (c_{i j} \in \mathbb{C}, i, j = 1, \dotsc, n).
  \end{equation*}
  Then $a \cdot \overline{v_j} = \overline{c_{1 j}} \cdot \overline{v_1} + \dotsb + \overline{c_{n j}} \overline{v_n}$ in $\overline{V}$. Hence,
  \begin{equation*}
    \chi_{\overline{V}}(a)
    = \overline{c_{1 1}} + \dotsb + \overline{c_{n n}}
    = \overline{c_{1 1} + \dotsb + c_{n n}} = \overline{\chi_V(\overline{a})}.
  \end{equation*}
  (2) Let $V \in \fdMod{A}$ be a simple module. If $a = x + \mathbf{i} y$ ($x, y \in A_0$), then, by (1),
  \begin{equation*}
    \chi_V(a) = \chi_V(x) + \mathbf{i} \chi_V(y),
    \quad \chi_{\overline{V}}(a)
    = \overline{\chi_V(x) - \mathbf{i} \chi_V(y)}.
  \end{equation*}
  From this, we see that $\chi_V(A_0) \subset \mathbb{R}$ if and only if $\chi_V = \chi_{\overline{V}}$. The latter statement is equivalent to $V \cong \overline{V}$ since both $V, \overline{V} \in \fdMod{A}$ are simple.
\end{proof}

Now we give characterizations of the $\mathcal{J}$-signature:

\begin{theorem}
  \label{thm:RF-sign-1}
  Let $V \in \fdMod{A}$ be a simple module. Then:
  \begin{enumerate}
  \item $\sigma_{\mathcal{J}}(V) = +1$ if and only if $V$ admits a basis such that the corresponding matrix representation $\rho: A \to M_n(\mathbb{C})$ has the following property:
    \begin{equation*}
      \rho(A_0) \subset M_n(\mathbb{R}).
    \end{equation*}
  \item $\sigma_{\mathcal{J}}(V) = -1$ if and only if $V$ does not admit such a basis but $\chi_V(A_0) \subset \mathbb{R}$.
  \item $\sigma_{\mathcal{J}}(V) = 0$ if and only if $\chi_V(A_0) \not \subset \mathbb{R}$.
  \end{enumerate}
\end{theorem}
\begin{proof}
  (1) and (3) follow immediately from Lemmas~\ref{lem:J-plus}, \ref{lem:J-plus-real-basis} and \ref{lem:J-zero}. Once (1) and (3) are proved, we find that the assertion (2) is equivalent to
  \begin{equation*}
    \sigma_{\mathcal{J}}(V) = -1 \iff \text{$\sigma_{\mathcal{J}}(V) \ne +1$ and $\sigma_{\mathcal{J}}(V) \ne 0$}.
  \end{equation*}
  This is obvious since $\sigma_{\mathcal{J}}(V) \in \{ 0, \pm 1 \}$. Thus (2) is proved.
\end{proof}

\begin{theorem}
  \label{thm:RF-sign-2}
  For a simple module $V \in \fdMod{A}$,
  \begin{equation}
    \label{eq:RF-sign-2-claim}
    \End_{A_0}(V) \cong
    \begin{cases}
      M_2(\mathbb{R}) & \text{if $\sigma_{\mathcal{J}}(V) = + 1$}, \\
      \mathbb{C}      & \text{if $\sigma_{\mathcal{J}}(V) = \ 0$}, \\
      \mathbb{H}      & \text{if $\sigma_{\mathcal{J}}(V) = - 1$}.
    \end{cases}
  \end{equation}
\end{theorem}
\begin{proof}
  We first show that one and only one of the following holds:
  \begin{enumerate}
  \item [(E1)] $V$ has a real form and $E := \End_{A_0}(V) \cong M_2(\mathbb{R})$.
  \item [(E2)] $V$ is simple as an $A_0$-module and $E \cong \mathbb{C}$.
  \item [(E3)] $V$ is simple as an $A_0$-module and $E \cong \mathbb{H}$.
  \end{enumerate}
  To see this, let $V_0$ be a simple $A_0$-submodule of $V$. Then, since $V_0 + \mathbf{i} V_0$ is closed under the action of $A$, $V = V_0 + \mathbf{i} V_0$. Since $V_0$ is a simple $A_0$-submodule, and since $\mathbf{i} V_0 \subset V$ is also an $A_0$-submodule, $V_0 \cap \mathbf{i} V_0$ is either one of $V_0$ or $\{ 0 \}$.

  If the case is the former, then $V = V_0$ is simple as an $A_0$-submodule. Hence $E$ is isomorphic to either one of $\mathbb{R}$, $\mathbb{C}$ or $\mathbb{H}$ by Schur's lemma and the classification of finite-dimensional division algebras over $\mathbb{R}$. However, since the map
  \begin{equation*}
    i: V \to V, \quad v \mapsto \mathbf{i} v \quad (v \in V)
  \end{equation*}
  is an element of $E$ such that $i^2 = -1$, $E$ cannot be isomorphic to $\mathbb{R}$. Therefore either one of (E2) or (E3) holds.

  If the case is the latter, then $V = V_0 + \mathbf{i} V_0$ is a direct sum of $V_0$ and $\mathbf{i} V_0$. Hence, in particular, $V_0$ is a real form of $V$. Since $\mathbf{i} V_0 \cong V_0$ as $A_0$-modules,
  \begin{equation*}
    E = \End_{A_0}(V_0 \oplus \mathbf{i} V_0) \cong \End_{A_0}(V_0 \oplus V_0) \cong M_2(D),
  \end{equation*}
  where $D = \End_{A_0}(V_0)$. Since $V_0$ is simple as an $A_0$-module, $D$ is isomorphic to either one of $\mathbb{R}$, $\mathbb{C}$ or $\mathbb{H}$. On the other hand,
  \begin{equation*}
    \End_A(V) = \{ f \in \End_{A_0}(V) \mid i \circ f = f \circ i \}
    \cong \left\{ \begin{pmatrix} a & b \\ -b & a \end{pmatrix} \mid a, b \in D \right\}
  \end{equation*}
  via the above isomorphism. Recalling $\End_A(V) \cong \mathbb{C}$, we conclude that $D$ must be isomorphic to $\mathbb{R}$. Hence (E1) holds.

  Now we prove \eqref{eq:RF-sign-2-claim}. If $\sigma_{\mathcal{J}}(V) = +1$, then, by Lemma~\ref{lem:J-plus}, $V$ is not simple as an $A_0$-module. Hence (E1) is the only possibility. Conversely, if (E1) holds, then $\sigma_{\mathcal{J}}(V) = +1$ again by Lemma~\ref{lem:J-plus}. Summarizing:
  \begin{equation}
    \label{eq:RF-sign-2-E1}
    \sigma_{\mathcal{J}}(V) = +1 \iff \text{(E1) holds}.
  \end{equation}

  If $\sigma_{\mathcal{J}}(V) = -1$, then, by Lemma~\ref{lem:J-minus}, $E$ has $\mathbb{H}$ as a subalgebra. Hence (E3) is the only possibility. Suppose, conversely, that (E3) holds. Fix an isomorphism $\phi: E \to \mathbb{H}$ of $\mathbb{R}$-algebras. We put $u = \phi(i)$ and define $v \in \mathbb{H}$ to be a pure imaginary quaternion such that $\langle u, v \rangle = \langle v, v \rangle = 1$, where $\langle,\rangle$ is given by
  \begin{align*}
    \langle a_1 + a_2 \mathbf{i} + a_3 \mathbf{d} + a_4 \mathbf{k},
    b_1 + b_2 \mathbf{i} + b_3 \mathbf{d} + b_4 \mathbf{k} \rangle
    = a_1 b_1 + a_2 b_2 + a_3 b_3 + a_4 b_4
  \end{align*}
  for $a_i, b_i \in \mathbb{R}$. Then one can verify that $u^2 = v^2 = -1$ and $v u = - u v$. This implies that $j = \phi^{-1}(v)$ is a $\mathcal{J}_-$-structure for $V$. Hence $\sigma_{\mathcal{J}}(V) = -1$. Summarizing:
  \begin{equation}
    \label{eq:RF-sign-2-E3}
    \sigma_{\mathcal{J}}(V) = -1 \iff \text{(E3) holds}.
  \end{equation}

  Finally, we consider the case where $\sigma_{\mathcal{J}}(V) = 0$. Then, by~\eqref{eq:RF-sign-2-E1} and~\eqref{eq:RF-sign-2-E3}, (E2) is the only possibility. Hence, $E \cong \mathbb{C}$. The proof is completed.
\end{proof}

We end this section by introducing the following terminologies:

\begin{definition}
  Let $A$ be a $\mathbb{C}$-algebra with real form $A_0$, and let $\mathcal{J}$ be the real structure for $\fdMod{A}$ associated with $A_0$. Given a simple module $V \in \fdMod{A}$, we say that $V$ is {\em real}, {\em complex}, and {\em quaternionic} with respect to $A_0$ if $\sigma_{\mathcal{J}}(V)$ is equal to $+1$, $0$ and $-1$, respectively.
\end{definition}

\section{Frobenius-Schur theorem for $C^*$-categories}
\label{sec:FS-thm-C-star-cat}

\subsection{$C^*$-categories}

A {\em $*$-category} is a $\mathbb{C}$-linear category $\mathcal{A}$ equipped with an anti-linear contravariant functor $*: \mathcal{A} \to \mathcal{A}$ such that $X^* = X$ for all object $X \in \mathcal{A}$ and $f^{**} = f$ for all morphism $f$ in $\mathcal{A}$. A {\em $C^*$-category} \cite{MR808930} is a $*$-category $\mathcal{A}$ satisfying the following three conditions:
\begin{enumerate}
\item[(1)] $\Hom_{\mathcal{A}}(X, Y)$ is a Banach space for all $X, Y \in \mathcal{A}$ and $\| f g \| \le \| f \| \| g \|$ holds for all composable morphisms $f$ and $g$ in $\mathcal{A}$.
\item[(2)] The $C^*$-identity $\| f^* f \| = \| f \|^2$ holds for all morphisms $f$ in $\mathcal{A}$.
\item[(3)] For any morphism $f$ in $\mathcal{A}$, the morphism $f^* f$ is positive.
\end{enumerate}
Note that, by (1) and (2), $\End_{\mathcal{A}}(X)$ is a $C^*$-algebra for each $X \in \mathcal{A}$. In (3), that $f^* f$ is positive means that it is a positive element of the $C^*$-algebra $\End_{\mathcal{A}}(X)$, where $X$ is the source of the morphism $f$.

We prepare some notations for $C^*$-algebras and $C^*$-categories. Given a Hausdorff space $X$, we denote by $C(X)$ the set of ($\mathbb{C}$-valued) continuous functions on $X$. Note that if $X$ is compact, then $C(X)$ is a $C^*$-algebra with the supremum norm. Now let, in general, $A$ be a unital $C^*$-algebra. If $a \in A$ is a normal element, then there exists a unique unit-preserving $*$-homomorphism $\phi_a: C(\mathrm{sp}(a)) \to A$, where $\mathrm{sp}(a)$ is the spectrum of $a$, mapping the inclusion map $\mathrm{sp}(a) \hookrightarrow \mathbb{C}$ to $a \in A$. For a subset $K \subset \mathbb{C}$ such that $\mathrm{sp}(a) \subset K$, we consider the map
\begin{equation*}
  \begin{CD}
    \tilde{\phi}_a: C(K) @>\text{restriction}>> C(\mathrm{sp}(a)) @>{\phi_a}>> A.
  \end{CD}
\end{equation*}
Given $f \in C(K)$, we write $f(a)$ for the element $\tilde{\phi}_a(f) \in A$ (the {\em continuous functional calculus}). In particular, the following notation will be used: For a positive element $a \in A$ and $\lambda > 0$, we write $a^\lambda$ for $f(a)$ with $f_{\lambda}(t) = t^\lambda$ ($t \ge 0$). If, moreover, $a$ is invertible, then $a^\lambda$ is defined for all $\lambda \in \mathbb{R}$ in a similar way.

Now let $\mathcal{A}$ be a $C^*$-category. A morphism $u$ in $\mathcal{A}$ is said to be {\em unitary} if it is invertible and $u^* = u^{-1}$. Given a morphism $f: X \to Y$ in $\mathcal{A}$, we define its {\em absolute value} by $|f| = (f^* f)^{1/2}$. $|f|$ is a positive element of the $C^*$-algebra $\End_{\mathcal{A}}(X)$. The following lemma will be used extensively:

\begin{lemma}
  \label{lem:C-star-cat-1}
  $f |f|^{-1}$ is unitary and $f |f|^{-1} = |f^*|^{-1} f$.
\end{lemma}

Let $\mathcal{C}$ be an arbitrary category, let $F: \mathcal{C} \to \mathcal{A}$ be an arbitrary functor, and let $\xi: F \to F$ be a natural transformation such that $\xi_X: F(X) \to F(X)$ is positive for all $X \in \mathcal{C}$. Given a continuous function $f: \mathbb{R}_{\ge 0} \to \mathbb{C}$, we can define $f(\xi_X)$ for each $X \in \mathcal{C}$. If $f$ is a polynomial function, then the family
\begin{equation*}
  f(\xi) := \{ f(\xi_X): F(X) \to F(X) \}_{X \in \mathcal{C}}
\end{equation*}
is obviously a natural transformation $f(\xi): F \to F$. By the Weierstrass approximation theorem, one can prove:

\begin{lemma}
  \label{lem:C-star-cat-2}
  $f(\xi)$ is a natural transformation for all $f \in C(\mathbb{R}_{\ge 0})$.
\end{lemma}

In particular, $\xi^{1/2} := \{ \xi_X^{1/2} \}_{X \in \mathcal{C}}$ is.

\subsection{Dual structures versus real structures}
\label{subsec:real-vs-dual}

We have introduced the notions of dual structures and real structures for a $\mathbb{C}$-linear category. For a $C^*$-category, it would be reasonable to require them to be compatible with the $*$-structure:

\begin{definition}
  \label{def:star-compatible}
  Let $\mathcal{F} = (F, u)$ be a dual structure or a real structure for a $C^*$-category $\mathcal{A}$. We say that $\mathcal{F}$ is {\em $*$-compatible} if $F$ is a $*$-functor ({\it i.e.}, $F(f^*) = F(f)^*$ for all morphism $f$ in $\mathcal{A}$) and $u$ is unitary ({\it i.e.}, $u_X^* = u_X^{-1}$ for all $X \in \mathcal{A}$).
\end{definition}

Given a $*$-compatible real structure $\mathcal{J} = (J, i)$ for $\mathcal{A}$, we put $D = J *$ and $\eta = i$. Since $D D = J * J * = J J * * = J J$, $\eta$ is a natural isomorphism from $\id_{\mathcal{A}}$ to $D D$. Moreover, $D* = * D$, $\eta^* = \eta^{-1}$, and
\begin{equation*}
  D(\eta_X) \circ \eta_{D(X)} = J(i_X^*) \circ i_{J(X)} = J(i_X^{-1}) J(i_X) = \id_{J(X)} = \id_{D(X)}
\end{equation*}
for all $X \in \mathcal{A}$. In other words, the pair $\mathbb{D}(\mathcal{J}) := (D, \eta)$ is a $*$-compatible dual structure for $\mathcal{A}$. Conversely, if $\mathcal{D} = (D, \eta)$ is a $*$-compatible dual structure for $\mathcal{A}$, then one can verify that the pair $\mathbb{J}(\mathcal{D}) = (D *, \eta)$ is a $*$-compatible real structure for $\mathcal{A}$. Since $** = \id_{\mathcal{A}}$, $\mathbb{D}$ and $\mathbb{J}$ are mutually inverse.

The main claim of this subsection is that the bijections $\mathbb{D}$ and $\mathbb{J}$ preserve the equivalence relations. Namely, if we denote by $\mathsf{DS}^*(\mathcal{A})$ ({\em resp}. $\mathsf{RS}^*(\mathcal{A})$) the class of equivalence classes of $*$-compatible dual ({\em resp}. real) structures for $\mathcal{A}$, then:

\begin{theorem}
  \label{thm:real-vs-dual}
  $\mathcal{J} \mapsto \mathbb{D}(\mathcal{J})$ induces a well-defined bijection
  \begin{equation}
    \label{eq:real-vs-dual-1}
    \mathbb{D}: \mathsf{RS}^*(\mathcal{A}) \to \mathsf{DS}^*(\mathcal{A}),
    \quad [\mathcal{J}] \mapsto [\mathbb{D}(\mathcal{J})]
  \end{equation}
  with the well-defined inverse
  \begin{equation}
    \label{eq:real-vs-dual-2}
    \mathbb{J}: \mathsf{DS}^*(\mathcal{A}) \to \mathsf{RS}^*(\mathcal{A}),
    \quad [\mathcal{D}] \mapsto [\mathbb{J}(\mathcal{D})].
  \end{equation}
\end{theorem}

We provide two lemmas to prove this theorem. We say that real structures $\mathcal{J}$ and $\mathcal{J}'$ for $\mathcal{A}$ are {\em unitary equivalent} if there exists an isomorphism $\beta: \mathcal{J} \to \mathcal{J}'$ of real structures such that $\beta_X$ is unitary for all $X \in \mathcal{A}$.

\begin{lemma}
  \label{lem:unitary-eq-real}
  For two $*$-compatible real structures $\mathcal{J} = (J, i)$ and $\mathcal{J}' = (J', i')$ for $\mathcal{A}$, the following assertions are equivalent:
  \begin{enumerate}
  \item $\mathcal{J}$ and $\mathcal{J}'$ are equivalent.
  \item $\mathcal{J}$ and $\mathcal{J}'$ are unitary equivalent.
  \end{enumerate}
\end{lemma}
\begin{proof}
  The implication $(2) \Rightarrow (1)$ is trivial. To prove $(1) \Rightarrow (2)$, let $\beta: \mathcal{J} \to \mathcal{J}'$ be an isomorphism of real structures. By Lemmas \ref{lem:C-star-cat-1} and \ref{lem:C-star-cat-2},
  \begin{equation*}
    u := \beta \circ |\beta|^{-1} = |\beta^*|^{-1} \circ \beta
  \end{equation*}
  is a unitary natural isomorphism from $J$ to $J'$. We show that $u$ is a morphism of real structures from $\mathcal{J}$ to $\mathcal{J}'$. Let $X \in \mathcal{A}$. By \eqref{eq:real-str-2} and the $*$-compatibility,
  \begin{equation}
    \label{eq:unitary-eq-real-1}
    \beta_{J(X)}^{} \beta_{J(X)}^*
    = J'(\beta_X)^{-1} (J'(\beta_X)^{-1})^{*}
    = J'(\beta_X^* \beta_X^{})^{-1}.
  \end{equation}
  Let, in general, $a: V \to V$ be a morphism in $\mathcal{A}$. If $a \ge 0$, then $a = b^* b$ for some $b \in \End_{\mathcal{A}}(V)$ and thus $J'(a) = J'(b)^* J'(b) \ge 0$. Hence, by the uniqueness of the positive square root, $J'(a^{1/2}) = J'(a)^{1/2}$. Applying this formula to \eqref{eq:unitary-eq-real-1}, we obtain $|\beta_{J(X)}^*| = J'(|\beta_X|)^{-1}$. Now we show that $u: \mathcal{J} \to \mathcal{J}'$ is a morphism of real structures as follows:
  \begin{equation*}
    J'(u_X) u_{J(X)} i_X = J'(\beta_X) J'(|\beta_X|^{-1}) |\beta_{J(X)}^*|^{-1} \beta_{J(X)} i_X = i'_X. \qedhere
  \end{equation*}
\end{proof}

Unitary equivalence of dual structures is defined in the same way as unitary equivalence of real structures. The proof of the following lemma is parallel to that of Lemma~\ref{lem:unitary-eq-real}. 

\begin{lemma}
  \label{lem:unitary-eq-dual}
  For two $*$-compatible dual structures $\mathcal{D} = (D, \eta)$ and $\mathcal{D}' = (D', \eta')$ for $\mathcal{A}$, the following assertions are equivalent:
  \begin{enumerate}
  \item $\mathcal{D}$ and $\mathcal{D}'$ are equivalent.
  \item $\mathcal{D}$ and $\mathcal{D}'$ are unitary equivalent.
  \end{enumerate}
\end{lemma}
\begin{proof}
  The implication $(2) \Rightarrow (1)$ is trivial. To prove $(1) \Rightarrow (2)$, let $\xi: \mathcal{D} \to \mathcal{D}'$ be an isomorphism. We show that the unitary natural isomorphism
  \begin{equation*}
    u := \xi \circ |\xi|^{-1} = |\xi^*|^{-1} \circ \xi: D \to D'
  \end{equation*}
  is a morphism $\mathcal{D} \to \mathcal{D}'$. Let $X \in \mathcal{A}$. By~\eqref{eq:dual-str-1} and the $*$-compatibility,
  \begin{equation*}
    \xi_{D(X)}^{} \circ \xi_{D(X)}^*
    = D'(\xi_X) \circ D'(\xi_X)^*
    = D'(\xi_X^* \circ \xi_X).
  \end{equation*}
  Taking the positive square root of both sides, we obtain $|\xi_{D(X)}^*| = D'(|\xi_X|)$ ({\it cf}. the proof of Lemma~\ref{lem:unitary-eq-real}). Hence,
  \begin{equation*}
    u_{D(X)} \eta_X = |\xi^*_{D(X)}|^{-1} \xi_{D(X)} \eta_X
    = D'(|\xi_X|)^{-1} D'(\xi_X) \eta_X'
    = D'(u_X) \eta_X'. \qedhere
  \end{equation*}
\end{proof}

\begin{proof}[Proof of Theorem~\ref{thm:real-vs-dual}]
  Let $\mathcal{J} = (J, i)$ and $\mathcal{J}' = (J', i')$ be $*$-compatible real structures for $\mathcal{A}$, and suppose that there exists an isomorphism $\beta: \mathcal{J} \to \mathcal{J}'$ of real structures. By Lemma~\ref{lem:unitary-eq-real}, we may assume that $\beta$ is unitary. For simplicity, we put $(D, \eta) = \mathbb{D}(\mathcal{J})$ and $(D', \eta') = \mathbb{D}(\mathcal{J}')$. Then $\beta$ is a natural isomorphism from $D$ to $D'$. Moreover, by \eqref{eq:real-str-2},
  \begin{equation*}
    D'(\beta_X) \eta'_X
    = J'(\beta_X^*) i'_X
    = J'(\beta_X)^{-1} i'_x
    = \beta_{J(X)} i_X
    = \beta_{D(X)} \eta_X
  \end{equation*}
  for all $X \in \mathcal{A}$. This means that $\beta: \mathbb{D}(\mathcal{J}) \to \mathbb{D}(\mathcal{J}')$ is an isomorphism of dual structures. Hence \eqref{eq:real-vs-dual-1} is well-defined. In a similar way, Lemma~\ref{lem:unitary-eq-dual} shows that \eqref{eq:real-vs-dual-2} is well-defined. It is obvious that \eqref{eq:real-vs-dual-1} and \eqref{eq:real-vs-dual-2} are mutually inverse.
\end{proof}

\subsection{Frobenius-Schur theorem}

Let $\mathcal{A}$ be a locally finite-dimensional abelian $C^*$-category. In \S\ref{sec:dual-structures}, we have introduced the Frobenius-Schur indicator $\nu_{\mathcal{D}}: \Obj(\mathcal{A}) \to \mathbb{Z}$ with respect to a dual structure $\mathcal{D}$ for $\mathcal{A}$. By Lemma~\ref{lem:FS-ind-equiv}, the assignment $\mathcal{D} \mapsto \nu_{\mathcal{D}}$ induces a map
\begin{equation*}
  \nu: \mathsf{DS}^*(\mathcal{A}) \to \mathbb{Z}^{\Obj(\mathcal{A})},
  \quad [\mathcal{D}] \mapsto \nu_{\mathcal{D}}.
\end{equation*}
In \S\ref{sec:real-structures}, we have defined the $\mathcal{J}$-signature $\sigma_{\mathcal{J}}: \Irr(\mathcal{A}) \to \mathbb{Z}$ for each real structure $\mathcal{J}$ for $\mathcal{A}$. By Lemma~\ref{lem:real-str-2}, $\mathcal{J} \mapsto \sigma_{\mathcal{J}}$ induces a map
\begin{equation*}
  \sigma: \mathsf{RS}^*(\mathcal{A}) \to \mathbb{Z}^{\Irr(\mathcal{A})},
  \quad [\mathcal{J}] \mapsto \sigma_{\mathcal{J}}.
\end{equation*}
Now we consider the diagram
\begin{equation*}
  \begin{CD}
    \mathsf{RS}^*(\mathcal{A}) @>{\sigma}>{}> \mathbb{Z}^{\Irr(\mathcal{A})} \\
    @V{\mathbb{D}}V{}V @A{}A\text{restriction}A \\
    \mathsf{DS}^*(\mathcal{A}) @>{}>{\nu}> \mathbb{Z}^{\Obj(\mathcal{A})}, \\
  \end{CD}
\end{equation*}
where $\mathbb{D}: \mathsf{RS}^*(\mathcal{A}) \to \mathsf{DS}^*(\mathcal{A})$ is the bijection of Theorem~\ref{thm:real-vs-dual}. We formulate the {\em Frobenius-Schur theorem for $C^*$-categories} as the commutativity of the above diagram. Namely, there holds:

\begin{theorem}
  \label{thm:FS-thm-C-st}
  Let $\mathcal{J} = (J, i)$ be a $*$-compatible real structure for $\mathcal{A}$. If $\mathcal{D} = (D, \eta)$ is a $*$-compatible dual structure for $\mathcal{A}$ such that $\mathcal{D} \sim \mathbb{D}(\mathcal{J})$, then
  \begin{equation}
    \label{eq:FS-thm-C-st-1}
    \sigma_{\mathcal{J}}(X) = \nu_{\mathcal{D}}(X)
  \end{equation}
  holds for all $X \in \Irr(\mathcal{A})$.
\end{theorem}
\begin{proof}
  Let $X \in \Irr(\mathcal{A})$. By Lemma \ref{lem:FS-ind-equiv}, we may assume $\mathcal{D} = \mathbb{D}(\mathcal{J})$. Then
  \begin{equation}
    \label{eq:FS-thm-C-st-2}
    \nu_{\mathcal{D}}(X) \ne 0
    \mathop{\iff} X \cong D(X)
    \mathop{\iff} X \cong J(X) \iff \sigma_{\mathcal{J}}(X) \ne 0.
  \end{equation}
  Here, the first equivalence follows from~\eqref{eq:FS-ind-simple-1}, the second from $D(X) = J(X)$ (the equality of objects of $\mathcal{A}$), and the last from the definition of $\sigma_{\mathcal{J}}$. In particular, \eqref{eq:FS-thm-C-st-1} holds in the case where $\sigma_{\mathcal{J}}(X) = 0$.

  Now we consider the case where $\varepsilon := \sigma_{\mathcal{J}}(X) \ne 0$. Put $\varepsilon' := \nu_{\mathcal{D}}(X)$. By \eqref{eq:FS-thm-C-st-2}, $\varepsilon, \varepsilon' \in \{ \pm 1 \}$. Let $j: X \to J(X)$ be a $\mathcal{J}_{\varepsilon}$-structure for $X$. Since $J(X) = D(X)$, $j$ is a morphism from $X$ to $D(X)$. By \eqref{eq:FS-ind-simple-2} with $f = j$,
  \begin{equation}
    \label{eq:FS-thm-C-st-4}
    \varepsilon' \cdot j^*
    = (D(j) \eta_X)^*
    = \eta_X^* D(j)^*
    = i_X^* J(j^*)^*
    = i_X^{-1} J(j).
  \end{equation}
  Hence we have
  \begin{equation}
    \label{eq:FS-thm-C-st-5}
    \varepsilon \varepsilon' \cdot j^* j
    = \varepsilon \cdot i_X^{-1} J(j) j
    = \varepsilon^2 \cdot i_X^{-1} i_X^{} = \id_X.
  \end{equation}
  Since $j^* j$ and $\id_X$ are positive elements of $\End_{\mathcal{A}}(X)$, $\varepsilon \varepsilon' \ge 0$. On the other hand, we have seen that $\varepsilon, \varepsilon' \in \{ \pm 1 \}$. Therefore $\varepsilon = \varepsilon'$, {\it i.e.}, \eqref{eq:FS-thm-C-st-1} holds.
\end{proof}

\section{Finite-dimensional $C^*$-algebras}
\label{sec:FS-thm-fd-C-star}

\subsection{Conventions}

Let $X$ and $Y$ be Hilbert spaces. We denote by $\mathcal{B}(X, Y)$ the set of all bounded linear operators from $X$ to $Y$. The following notations will be used:
\begin{equation*}
  \mathcal{B}(X) := \mathcal{B}(X, X), \quad X^\vee = \mathcal{B}(X, \mathbb{C}).
\end{equation*}
For $x \in X$, we put $\phi_x = \langle x | - \rangle$, where $\langle|\rangle$ is the inner product on $X$. Our convention is that the inner product on a Hilbert space is anti-linear in the first variable and linear in the second. Thus $\phi_x \in X^\vee$ for all $x \in X$ and the map
\begin{equation}
  \label{eq:Riesz-1}
  \phi_X: X \to X^\vee, \quad x \mapsto \phi_x \quad (x \in X)
\end{equation}
is anti-linear. The Riesz representation theorem  states that $\phi_X$ is bijective. Moreover, $\phi_X$ is unitary if we define an inner product on $X^\vee$ by
\begin{equation}
  \label{eq:ip-dual-Hilb-sp}
  \langle \phi_x | \phi_y \rangle_{X^\vee} = \langle y | x \rangle
  \quad (x, y \in X).
\end{equation}
The complex conjugate $\overline{X}$ is also a Hilbert space with the inner product given by $\langle \overline{x} | \overline{y} \rangle_{\overline{X}} = \overline{\langle x | y \rangle}$ for $x, y \in X$. Let $\mathcal{H}$ be the $C^*$-category of Hilbert spaces. The following category-theoretical interpretation of Riesz's theorem will be important:

\begin{lemma}
  \label{lem:Riesz-1}
  Given $X \in \mathcal{H}$, we denote by $\varphi_X$ the map $\phi_X: X \to X^\vee$ regarded as a linear map $\overline{X} \to X^\vee$. Then $\varphi = \{ \varphi_X \}_{X \in \mathcal{H}}$ is a natural isomorphism
  \begin{equation*}
    \varphi: \overline{\phantom{a}} \circ * \to (-)^\vee,
  \end{equation*}
  where $*: \mathcal{H} \to \mathcal{H}$ is the functor taking the adjoint operator, $\overline{\phantom{a}}: \mathcal{H} \to \mathcal{H}$ is taking the complex conjugate and $(-)^\vee: \mathcal{H} \to \mathcal{H}$ is taking the continuous dual.
\end{lemma}
\begin{proof}
  We interpret our claim in terms of \eqref{eq:Riesz-1} and then find that the claim is equivalent to that the equation $\phi_X \circ f^* = f^\vee \circ \phi_Y$ holds for all morphism $f: X \to Y$ in $\mathcal{H}$. This can be verified straightforward.
\end{proof}

\subsection{Real forms of a $*$-algebra}
\label{subsec:RF-st-alg}

Let $A$ be a $*$-algebra. By a {\em $*$-representation} of $A$, we mean a Hilbert space $X$ endowed with a $*$-homomorphism $A \to \mathcal{B}(X)$. This is the same thing as a Hilbert space $X$ endowed with a left $A$-module structure such that $x \mapsto a x$ ($x \in X$) is bounded for all $a \in A$ and
\begin{equation*}
  \langle x | a y \rangle = \langle a^* x | y \rangle
  \quad (x, y \in X, a \in A).
\end{equation*}
We denote by $\Rep(A)$ the category whose objects are $*$-representations of $A$ and whose morphisms are bounded $A$-linear maps between them. $\fdRep(A)$ denotes the full subcategory of $\Rep(A)$ consisting of finite-dimensional objects.

\begin{proposition}
  $\Rep(A)$ and $\fdRep(A)$ are $C^*$-categories.
\end{proposition}

Recall from \S\ref{subsec:RF-alg} that a {\em real form} of $A$ is an $\mathbb{R}$-subalgebra $A_0 \subset A$ such that $A = A_0 \oplus \mathbf{i} A_0$. Given a real form $A_0$ of $A$, we define $S: A \to A$ by
\begin{equation}
  \label{eq:RF-st-alg-1}
  S(a) = (\overline{a})^* \quad (a \in A),
\end{equation}
where $\overline{a}$ is the conjugate of $a \in A$ with respect to $A_0$. Since $\overline{\phantom{a}}$ and $*$ are anti-linear maps, $S$ is linear. Moreover, by~\eqref{eq:RF-cpx-conj-2}, we obtain:
\begin{equation}
  \label{eq:RF-st-alg-2}
  S(a b) = S(b) S(a), \quad S(S(a)^*)^* = a \quad (a, b \in A).
\end{equation}
Following, we refer the map $S$ as the {\em anti-algebra map associated with the real form $A_0$}. Conversely, if such a linear map $S$ is given, then
\begin{equation}
  \label{eq:RF-st-alg-3}
  A_0 = \{ a \in A \mid \overline{a} = a \}, \text{\quad where $\overline{a} = S(a)^*$},
\end{equation}
is a real form of $A$, which will be referred to as the {\em real form associated with $S$}. It is easy to see that \eqref{eq:RF-st-alg-1} and~\eqref{eq:RF-st-alg-3} establish a bijection between real forms of $A$ and linear maps $S$ satisfying~\eqref{eq:RF-st-alg-2}.

Note that we do not require a real form $A_0$ to be closed under the $*$-operation of $A$; if $S$ is the anti-algebra map associated with a real form $A_0$, then:
\begin{equation}
  \label{eq:RF-st-alg-4}
  \text{$a^* \in A_0$ for all $a \in A_0$}
  \iff S \circ * = * \circ S
  \iff S^2 = \id_A.
\end{equation}

We say that $x \in A$ is positive\footnote{Since we do not assume $A$ to be a $C^*$-algebra, we should clarify the meaning of the positivity of an element of $A$.} if $x = a^* a$ for some $a \in A$. Now we suppose that there exists a positive element $g \in A$ such that the pair $(S, g)$ is a dual structure for $A$ in the sense of \S\ref{subsec:DS-alg}. For example, we can choose $g$ to be $1$ if one of the equivalent conditions of \eqref{eq:RF-st-alg-4} is satisfied.

Let $X \in \Rep(A)$. Then its continuous dual $X^\vee$ has a left $A$-module structure given by the same formula as \eqref{eq:dual-module-1}. With respect to this action,
\begin{equation}
  \label{eq:Riesz-A-linear}
  \phi_{\overline{a} \cdot x} = a \cdot \phi_x \quad (a \in A, x \in X).
\end{equation}
In general, $X^\vee$ is {\em not} a $*$-representation of $A$ with respect to the standard inner product given by \eqref{eq:ip-dual-Hilb-sp}. Following, we define $D(X)$ to be the left $A$-module $X^\vee$ with the inner product given by
\begin{equation*}
  \langle \phi_x | \phi_y \rangle_{D(X)} = \langle y | g x \rangle
  \quad (x, y \in X).
\end{equation*}
The map $\phi_X: X \to D(X)$ is not unitary in general but is still bounded. By~\eqref{eq:Riesz-A-linear}, one can check that $D(X)$ is a $*$-representation of $A$. The assignment $X \mapsto D(X)$ defines a contravariant endofunctor on $\Rep(A)$. Now we define
\begin{equation*}
  \eta_X: X \to D D(X), \quad \langle \eta_X(x), \lambda \rangle = \langle \lambda, g x \rangle
  \quad (x \in X, \lambda \in X^\vee)
\end{equation*}

\begin{lemma}
  \label{lem:RF-st-alg-dual-str}
  $\mathcal{D} = (D, \eta)$ is a $*$-compatible dual structure for $\Rep(A)$.
\end{lemma}
\begin{proof}
  It is obvious that $\mathcal{D}$ is a dual structure for $\Rep(A)$ ({\it cf}. \S\ref{subsec:DS-alg}). We show that it is  $*$-compatible. If $f: X \to Y$ is a morphism in $\Rep(A)$, then
  \begin{gather*}
    \langle f^{\vee\ast} (\phi_x) | \phi_y \rangle_{D(Y)}
    = \langle \phi_x | f^\vee (\phi_y) \rangle_{D(Y)}
    = \langle \phi_x | \phi_{f^{\ast}(y)} \rangle
    = \langle f^{\ast}(y) | g x \rangle \\
    = \langle y | f(g x) \rangle
    = \langle y | g f^{\ast\ast}(x) \rangle
    = \langle \phi_y | \phi_{f^{\ast\ast}(x)} \rangle
    = \langle \phi_y | f^{*\vee}(\phi_x) \rangle
  \end{gather*}
  for all $x \in X$ and $y \in Y$. This means that $f^{\vee \ast} = f^{\ast \vee}$. To show that $\eta$ is unitary, we note that $\eta$ is expressed as
  \begin{equation}
    \label{eq:Riesz-bidual}
    \eta_X = \phi_{D(X)} \circ \phi_X.
  \end{equation}
  By~\eqref{eq:Riesz-A-linear} and~\eqref{eq:Riesz-bidual}, we can show that the bijection $\eta_X: X \to D D(X)$ preserves the inner product as follows: For $x, x' \in X$,
  \begin{equation*}
    \langle \eta_X(x) | \eta_X(x') \rangle_{D D(X)}
    = \langle \phi_x | g \phi_{x'} \rangle_{D(X)}
    = \langle \overline{g}x' | g x \rangle
    = \langle x' | (\overline{g})^* g x \rangle
    = \langle x' | x \rangle. \qedhere
  \end{equation*}
\end{proof}

In \S\ref{subsec:RF-alg}, we have introduced the notion of the conjugate $A$-module $\overline{X}$ with respect to the real form $A_0$. Given $X \in \Rep(A)$, we define $J(X)$ to be the left $A$-module $\overline{X}$ with the inner product given by $\langle \overline{x} | \overline{y} \rangle_{J(X)} = \langle y | g x \rangle$ ($x, y \in X$).

Equation \eqref{eq:Riesz-A-linear} means that the map $\varphi_X$ of Lemma~\ref{lem:Riesz-1} is an isomorphism $\varphi_X: J(X) \to D(X)$ of left $A$-modules. Recall that $D(X)$ is a $*$-representation. Since $\varphi_X$ preserves the inner product, $J(X)$ is also a $*$-representation of $A$. Now we define $i: \id \to J J$ by the same formula as~\eqref{eq:RF-conj-rep-4}. Then:

\begin{lemma}
  $\mathcal{J} = (J, i)$ is a $*$-compatible real structure for $\Rep(A)$.
\end{lemma}

We omit the proof since it can be proved in a similar way as Lemma~\ref{lem:RF-st-alg-dual-str}.

In \S\ref{subsec:real-vs-dual}, we have seen that the pair $\mathbb{D}(\mathcal{J}) = (J \circ *, i)$ is a $*$-compatible dual structure for $\Rep(A)$. Now we claim:

\begin{lemma}
  \label{lem:RF-real-dual-equiv}
  The natural isomorphism $\varphi$ of Lemma~\ref{lem:Riesz-1} induces a unitary equivalence of dual structures $\varphi: \mathbb{D}(\mathcal{J}) \to \mathcal{D}$.
\end{lemma}
\begin{proof}
  We have already seen that $\varphi_X: J(X) \to D(X)$ is a unitary isomorphism of $A$-modules for all $X \in \Rep(A)$. The naturality is trivial. To complete the proof, we need to show that $\varphi: \mathbb{D}(\mathcal{J}) \to \mathcal{D}$ is a morphism of dual structures, {\it i.e.},
  \begin{equation}
    \label{eq:RF-dual-real-equiv-1}
    \varphi_{J(X)} \circ i_X = D(\varphi_X) \circ \eta_X
    : X \to D J(X)
  \end{equation}
  holds for all $X \in \Rep(A)$. For $x, y \in X$,
  \begin{gather*}
    \langle \varphi_{J(X)} i_X(x), \overline{y} \rangle
    = \langle \varphi_{J(X)}(\overline{\overline{x}}), \overline{y} \rangle
    = \langle \overline{x} | \overline{y} \rangle_{J(X)}
    = \langle y | g x \rangle. \\
    \langle D(\varphi_X) \eta_X(x), \overline{y} \rangle
    = \langle \eta_X(x), \varphi_X(y) \rangle
    = \langle \varphi_X(y), g x \rangle
    = \langle y | g x \rangle.
  \end{gather*}
  Hence~\eqref{eq:RF-dual-real-equiv-1} is verified.
\end{proof}

By Theorem~\ref{thm:FS-thm-C-st} and Lemma~\ref{lem:RF-real-dual-equiv}, we obtain the following result: If $V$ is a finite-\hspace{0pt}dimensional irreducible $*$-representation of $A$, then
\begin{equation}
  \label{eq:FS-thm-C-st-0}
  \nu_{\mathcal{D}}(V) = \sigma_{\mathcal{J}}(V).
\end{equation}
Theorem \ref{thm:FS-thm-C-st} also implies that the $\mathcal{J}$-signature is equal to the Frobenius-Schur indicator with respect to the dual structure $\mathbb{D}(\mathcal{J})$. Unlike $\mathbb{D}(\mathcal{J})$, the dual structure $\mathcal{D}$ does not have an anti-linear part and is easy to deal with.

\begin{remark}
  The following direct proof of \eqref{eq:FS-thm-C-st-0} can be obtained by interpreting the proof of Theorem~\ref{thm:FS-thm-C-st} in our context: First observe that
  \begin{equation*}
    \sigma_{\mathcal{J}}(V) \ne 0 \iff V \cong J(V) \iff V \cong D(V) \iff \nu_{\mathcal{D}}(V) \ne 0
  \end{equation*}
  ({\it cf}. \eqref{eq:FS-thm-C-st-2} in the proof of Theorem~\ref{thm:FS-thm-C-st}). Hence~\eqref{eq:FS-thm-C-st-0} is proved in the case where $\varepsilon := \sigma_{\mathcal{J}}(V) = 0$. Now consider the case where $\varepsilon \ne 0$. Then there exists an anti-linear map $j: V \to V$ satisfying~\eqref{eq:RF-J-eps-str}. By using $j$, we define
  \begin{equation*}
    \beta(v, w) = \langle j(v) | w \rangle \quad (v, w \in V).
  \end{equation*}
  $\beta$ is a non-degenerate bilinear form on $V$ such that $\beta(a v, w) = \beta(v, S(a) w)$ for all $v, w \in V$. Hence, by \cite[Theorem 3.4]{KenichiShimizu:2012-11}, $\beta(w, g v) = \varepsilon' \cdot \beta(v, w)$ for all $v, w \in V$, where $\varepsilon' := \nu_{\mathcal{D}}(V)$. Now we compute:
  \begin{equation*}
    \langle w | g v \rangle
    = \varepsilon \cdot \langle j^2(w) | g v \rangle
    = \varepsilon \cdot \beta(j(w), g v)
    = \varepsilon \varepsilon' \cdot \beta(v, j(w))
    = \varepsilon \varepsilon' \cdot \langle j(v) | j(w) \rangle
  \end{equation*}
  for $v, w \in V$ ({\it cf}. \eqref{eq:FS-thm-C-st-4} and \eqref{eq:FS-thm-C-st-5}). Recall that $g = a^* a$ for some $a \in A$. For any non-zero element $v \in V$, we have:
  \begin{equation*}
    \varepsilon \varepsilon'
    = \varepsilon \varepsilon' \cdot \| j(v) \| \cdot \| j(v) \|^{-1}
    = \langle v | g v \rangle \cdot \| j(v) \|^{-1}
    = \| a v \| \cdot \| j(v) \|^{-1} \ge 0.
  \end{equation*}
  Since $\varepsilon, \varepsilon' \in \{ \pm 1 \}$, $\varepsilon = \varepsilon'$ follows. Thus~\eqref{eq:FS-thm-C-st-0} is proved.
\end{remark}

\begin{remark}[Lifting problem]
  \label{rem:lifting-problem}
  The positive element $g \in A$ plays a key role to define $\mathcal{J}$ and $\mathcal{D}$. Such an element does not always seem to exists. If existence of $g$ is not guaranteed, then we consider $X \mapsto \overline{X}$ as a functor
  \begin{equation*}
    J_0: \Rep(A) \to \Mod{A}, \quad X \mapsto \overline{X}
  \end{equation*}
  since we do not know how to make $\overline{X}$ into a $*$-representation of $A$. By the same reason, we consider $X \mapsto X^\vee$ as a contravariant functor
  \begin{equation*}
    D_0: \Rep(A) \to \Mod{A}, \quad X \mapsto X^\vee.
  \end{equation*}
  Let, in general, $F: \Rep(A) \to \Mod{A}$ be a (contravariant) functor. By a {\em lift} of $F$ on a full subcategory $\mathcal{C} \subset \Rep(A)$, we mean a (contravariant) endo-$*$-functor $\tilde{F}$ on $\mathcal{C}$ such that $U \circ \tilde{F} = F|_{\mathcal{C}}$, where $U: \mathcal{C} \to \Mod{A}$ is the functor forgetting the inner product. Under some technical assumptions, we can extend a lift of $D_0$ on $\mathcal{C}$ to a $*$-compatible dual structure for $\mathcal{C}$ (Proposition~\ref{prop:D0-lift}). In view of this fact, it would be important to study when a lift of $D_0$ exists. We discuss this problem in Appendix~A.
\end{remark}

\subsection{Frobenius-Schur theorem for finite-dimensional $C^*$-algebras}

We apply our results to finite-dimensional $C^*$-algebras. The following proposition is due to B\"ohm, Nill and Szlach\'anyi \cite{MR1726707}.

\begin{proposition}
  \label{prop:BNS}
  Let $A$ be a finite-dimensional $C^*$-algebra. For each linear map $S: A \to A$ satisfying~\eqref{eq:RF-st-alg-2}, there uniquely exists an invertible positive element $g \in A$ satisfying the following two conditions:
  \begin{enumerate}
  \item[(1)] $S^2(a) = g a g^{-1}$ for all $a \in A$.
  \item[(2)] $\chi_V(g) = \chi_V(g^{-1}) > 0$ for all simple left $A$-module $V$.
  \end{enumerate}
  Moreover, the element $g$ fulfills:
  \begin{enumerate}
  \item[(3)] $S(g) = g^{-1}$.
  \end{enumerate}
\end{proposition}

Namely, an element $g$ as in \S\ref{subsec:RF-st-alg} always exists and, moreover, can be chosen in a canonical way. In Appendix A, we will give another proof of Proposition~\ref{prop:BNS} from the viewpoint of the lifting problem mentioned in Remark~\ref{rem:lifting-problem}.

Now we give the {\em Frobenius-Schur theorem for finite-dimensional $C^*$-algebras} in a form as general as possible:

\begin{theorem}
  \label{thm:FS-thm-fd-C-st-alg}
  Let $A$ be a finite-dimensional $C^*$-algebra, let $A_0$ be a real form of $A$, and let $E = \sum_i E_i' \otimes E_i''$ be a separability idempotent of $A$. Given a simple left $A$-module $V$, we set
  \begin{equation*}
    \nu(V) = \sum_i \chi_V(S(E_i') g E_i''),
  \end{equation*}
  where $S: A \to A$ is the anti-algebra map associated with $A_0$ and $g \in A$ is the element given by Proposition~\ref{prop:BNS}. Then
  \begin{equation}
    \label{eq:FS-thm-fd-C-st-alg-claim}
    \nu(V) =
    \begin{cases}
      + 1 & \text{if $V$ is real}, \\
      \ 0 & \text{if $V$ is complex}, \\
      - 1 & \text{if $V$ is quaternionic with respect to the real form $A_0$}.
    \end{cases}
  \end{equation}
\end{theorem}
\begin{proof}
  We may assume that $V$ is a $*$-representation since, by the classification theorem of finite-dimensional $C^*$-algebras, there exists an isomorphism
  \begin{equation}
    \label{eq:classification-fd-C-st-alg}
    A \cong M_{n_1}(\mathbb{C}) \oplus \dotsb \oplus M_{n_r}(\mathbb{C})
  \end{equation}
  of $C^*$-algebras for some $n_1, \dotsc, n_r$.

  Define $\mathcal{D}$ and $\mathcal{J}$ as in \S\ref{subsec:RF-st-alg} by using the element $g$. By Theorem~\ref{thm:FS-ind-sep-alg}, the left-hand side of~\eqref{eq:FS-thm-fd-C-st-alg-claim} is equal to the Frobenius-Schur indicator $\nu_{\mathcal{D}}(V)$. On the other hand, by definition, the right-hand side of \eqref{eq:FS-thm-fd-C-st-alg-claim} is equal to the $\mathcal{J}$-signature $\sigma_{\mathcal{J}}(V)$. Thus~\eqref{eq:FS-thm-fd-C-st-alg-claim} follows from~\eqref{eq:FS-thm-C-st-0}.
\end{proof}

\begin{remark}
  \label{rem:fd-C-st-characterization}
  Since we will deal with examples which are not {\it a priori} a $C^*$-algebra, we mention the following characterization: For a finite-dimensional $*$-algebra $A$, the following assertions are equivalent:
  \begin{enumerate}
  \item $A$ admits a norm making it into a $C^*$-algebra.
  \item $A$ has a separability idempotent of the form $E = \sum_{i = 1}^m a_i^* \otimes a_i$, $a_i \in A$.
  \item Any finite-dimensional left $A$-module admits an inner product making it into a $*$-representation.
  \item $A$ has a faithful $*$-representation.
  \end{enumerate}
  To show $(1) \Rightarrow (2)$, let $e_{i j}^{(k)} \in A$ be the element corresponding to the $(i,j)$-th matrix unit of the $k$-th component $M_{n_k}(\mathbb{C})$ via the isomorphism~\eqref{eq:classification-fd-C-st-alg}. Then it is straightforward to show that the element
  \begin{equation*}
    E = \sum_{k = 1}^r \sum_{i, j = 1}^{n_k} \frac{1}{\sqrt{n_k}} e_{i j}^{(k)} \otimes \frac{1}{\sqrt{n_k}} e_{j i}^{(k)}
    \in A \otimes_{\mathbb{C}} A
  \end{equation*}
  is a separability idempotent of the desired form.

  The implication (2) $\Rightarrow$ (3) is shown as follows: Let $E = \sum a_i^* \otimes a_i$ be a separability idempotent of such form. Given $V \in \fdMod{A}$, we fix an inner product $\langle | \rangle_0$ on $V$ and then define
  \begin{equation*}
    \langle v | w \rangle = \sum_i \langle a_i v | a_i w \rangle_0 \quad (v, w \in V).
  \end{equation*}
  This makes $V$ into a $*$-representation. Thus (3) follows. To show (3) $\Rightarrow$ (4), make the left regular representation of $A$ into a $*$-representation. To show (4) $\Rightarrow$ (1), fix a faithful $*$-representation $X$ of $A$ and then realize $A$ as a $*$-subalgebra of the $C^*$-algebra $\mathcal{B}(X)$. The operator norm makes $A$ into a $C^*$-algebra.
\end{remark}

\begin{remark}[{\it cf.} Doi {\cite[\S1]{MR2674691}}]
  \label{rem:fd-C-st-sep-idemp}
  In applications, it is often a problem how to find a separability idempotent. Let $A$ be a finite-dimensional $C^*$-algebra. If we are given an inner product $\langle | \rangle$ on $A$ making it into a $*$-representation of $A$, then we can express a separability idempotent of $A$ as follows: First, fix an orthonormal basis $\{ e_i \}_{i=1}^m$ of $A$ with respect to $\langle|\rangle$. Consider the linear map
  \begin{equation*}
    \Theta: A \otimes A \to \End_{\mathbb{C}}(A),
    \quad \Theta(a \otimes b)(x) = \langle b^* | x \rangle \cdot a
    \quad (a, b, x \in A).
  \end{equation*}
  Since $\langle|\rangle$ is non-degenerate, $\Theta$ is bijection. Thus we have
  \begin{equation}
    \label{eq:fd-C-st-alg-dual-basis}
    \sum_{i = 1}^m a e_i \otimes e_i^* = \sum_{i = 1}^m e_i \otimes e_i^* a \quad (a \in A),
  \end{equation}
  since $\Theta(\sum_{i = 1}^m a e_i \otimes e_i^*)(x) = a x = \Theta(\sum_{i = 1}^m e_i \otimes e_i^* a)(x)$ for all $x \in A$. Now we set $v = \sum_{i = 1}^m e_i e_i^*$. It is obvious that $v$ is positive and central. Moreover, $v$ is invertible by \cite[Theorem~1.5]{MR2674691}. By~\eqref{eq:fd-C-st-alg-dual-basis}, we see that
  \begin{equation*}
    E = \sum_{i = 1}^m e_i \otimes e_i^* v^{-1}
    = \sum_{i = 1}^m a_i^* \otimes a_i, \quad \text{where $a_i = e_i^* v^{-1/2}$},
  \end{equation*}
  is a separability idempotent.
\end{remark}

\subsection{Example I. Weak Hopf $C^*$-algebras}

A {\em weak Hopf algebra} \cite{MR1726707,MR1793595} is an algebra $A$ over $\mathbb{C}$, which is a coalgebra $(A, \Delta, \varepsilon)$ over $\mathbb{C}$ at the same time, satisfying numerous axioms. To denote the comultiplication, we use Sweedler's notation:
\begin{equation*}
  \Delta(a) = a_{(1)} \otimes a_{(2)} \quad (a \in A).
\end{equation*}
We omit the detailed definition but note that a weak Hopf algebra $A$ has a special linear map $S: A \to A$ called the {\em antipode}. It is known that the antipode is an anti-algebra map and an anti-coalgebra map:
\begin{equation*}
  S(a b) = S(b) S(a),
  \quad S(a)_{(1)} \otimes S(a)_{(2)} = S(a_{(2)}) \otimes S(a_{(1)})
  \quad (a, b \in A).
\end{equation*}
A {\em weak Hopf $*$-algebra} \cite[\S 4.1]{MR1726707} is a weak Hopf algebra $A$ such that whose underlying algebra is a $*$-algebra and there holds
\begin{equation*}
  (a^*)_{(1)} \otimes (a^*)_{(2)} = (a_{(1)})^* \otimes (a_{(2)})^* \quad (a \in A).
\end{equation*}
If $A$ is a weak Hopf $*$-algebra, then $\varepsilon(a^*) = \overline{\varepsilon(a)}$ and $S(S(a)^*)^* = a$ for all $a \in A$. In particular, $S$ satisfies \eqref{eq:RF-st-alg-2}. The corresponding real form
\begin{equation*}
  A_0 := \{ a \in A \mid S(a)^* = a \}
\end{equation*}
will be referred to as the {\em canonical real form} of $A$.

Now let $A$ be a {\em finite-dimensional weak Hopf $C^*$-algebra} \cite[\S 4.1]{MR1726707}, {\it i.e.}, a weak Hopf algebra whose underlying $*$-algebra is finite-dimensional and satisfies the conditions of Remark~\ref{rem:fd-C-st-characterization}. Since the antipode $S$ of $A$ satisfies \eqref{eq:RF-st-alg-2}, $A$ has a unique element $g$ satisfying the conditions of Proposition~\ref{prop:BNS}. The element $g$ is called the {\em canonical grouplike element} \cite[\S 4.3]{MR1726707} since it satisfies
\begin{equation*}
  1_{(1)} g \otimes 1_{(2)} g = \Delta(g) = g 1_{(1)} \otimes g 1_{(2)},
\end{equation*}
where $1_{(1)} \otimes 1_{(2)} = \Delta(1_A)$. Define $\varepsilon_L, \varepsilon_R: A \to A$ by
\begin{equation*}
  \varepsilon_L(a) = \varepsilon(1_{(1)} a) 1_{(2)}, \quad \varepsilon_R(a) = \varepsilon(1_{(2)} a) 1_{(1)} \quad (a \in A).
\end{equation*}
A {\em Haar integral} \cite[\S3.6]{MR1726707} of $A$ is an element $\Lambda \in A$ such that
\begin{equation*}
  \varepsilon_L(\Lambda) = 1 = \varepsilon_R(\Lambda),
  \quad a \Lambda = \varepsilon_L(a) \Lambda,
  \quad \Lambda a = \Lambda \varepsilon_R(a)
  \quad (a \in A).
\end{equation*}
Existence and uniqueness of a Haar integral is proved in \cite[\S 4.2]{MR1726707}. Now let $\Lambda$ be the Haar integral of $A$. In the above notations, we propose the following {\em Frobenius-Schur theorem for weak Hopf $C^*$-algebras}:

\begin{theorem}
  \label{thm:FS-thm-WHCst}
  For a simple left $A$-module $V$,
  \begin{equation}
    \label{eq:FS-thm-WHCst}
    \chi_V(\Lambda_{(1)}\Lambda_{(2)}) = \frac{\chi_V(1)}{\chi_V(g)} \times \begin{cases}
      + 1 & \text{if $V$ is real}, \\
      \ 0 & \text{if $V$ is complex}, \\
      - 1 & \text{if $V$ is quaternionic} \\
    \end{cases}
  \end{equation}
  with respect to the canonical real form $A_0 = \{ a \in A \mid S(a)^* = a \}$.
\end{theorem}

Note that $\chi_V(1) = \dim_{\mathbb{C}}(V) > 0$. By Remark~\ref{rem:fd-C-st-characterization}, we may assume that $V$ is a $*$-representation. Hence, since $g$ is positive and invertible, $\chi_V(g) > 0$. From this theorem, we obtain the following result not involving $g$: The left-hand side of \eqref{eq:FS-thm-WHCst} is positive, zero or negative according to whether $V$ is real, complex or quaternionic with respect to the canonical real form.

One can prove Theorem~\ref{thm:FS-thm-WHCst} by applying Theorem~\ref{thm:FS-thm-fd-C-st-alg} to the canonical real form of $A$. We omit the detail since Theorem~\ref{thm:FS-thm-WHCst} is a special case of the following {\em twisted version} of Theorem~\ref{thm:FS-thm-WHCst}. By an {\em automorphism} of $A$, we mean an automorphism $\tau: A \to A$ of the underlying algebra such that
\begin{equation*}
  \tau(a^*) = \tau(a)^*,
  \quad \tau(a)_{(1)} \otimes \tau(a)_{(2)} = \tau(a_{(1)}) \otimes \tau(a_{(2)})
  \quad (a \in A).
\end{equation*}
If $\tau$ is an automorphism of $A$, then
\begin{equation*}
  \varepsilon \circ \tau = \varepsilon,
  \quad S \circ \tau = \tau \circ S,
  \quad \tau(g) = g,
  \quad \tau(\Lambda) = \Lambda
\end{equation*}
by their uniqueness. Now let $\tau$ be an {\em involution} of $A$, {\it i.e.}, an automorphism of $A$ such that $\tau^2 = \id_A$. Then:

\begin{theorem}
  \label{thm:FS-thm-WHCst-tw}
  For a simple left $A$-module $V$, we set
  \begin{equation*}
    \sigma_{\tau}(V) = \frac{\chi_V(g)}{\chi_V(1)} \cdot \chi_V(\tau(\Lambda_{(1)}) \Lambda_{(2)}).
  \end{equation*}
  Then $\sigma_{\tau}(V) \in \{ 0, \pm 1 \}$. Moreover, $\sigma_{\tau}(V)$ has the following meaning:
  \begin{enumerate}
  \item $\sigma_{\tau}(V) = +1$ if and only if $V$ admits a basis such that the matrix representation $\rho: A \to M_n(\mathbb{C})$ with respect to that basis fulfills:
    \begin{equation*}
      \rho(\tau(a)) = \overline{\rho(a)} \text{\quad for all $a \in A_0$}.
    \end{equation*}
  \item $\sigma_{\tau}(V) = -1$ if and only if $V$ does not admit such a basis but satisfies
    \begin{equation*}
      \chi_V(\tau(a)) = \overline{\chi_V(a)} \text{\quad for all $h \in A_0$}.
    \end{equation*}
  \item $\sigma_{\tau}(V) = 0$ if and only if $\chi_V(\tau(a)) \ne \overline{\chi_V(a)}$ for some $a \in A_0$.
  \end{enumerate}
\end{theorem}

Applying this theorem to the group algebra of a finite group, we obtain the result of Kawanaka and Matsuyama \cite{MR1078503}. This theorem is a generalization of their result to weak Hopf $C^*$-algebras ({\it cf}. Sage and Vega \cite{MR2879228}).

\begin{proof}
  Since $S^\tau := \tau \circ S$ is a linear map satisfying~\eqref{eq:RF-st-alg-2},
  \begin{equation*}
    A_0^\tau := \{ a \in A \mid S^\tau(a)^* = a \}
  \end{equation*}
  is a real form. We define $\tilde{\sigma}_{\tau}(V)$ to be $+1$, $0$ or $-1$ according to whether $V$ is real, complex or quaternionic with respect to the real form $A_0^{\tau}$. By the definition of the Haar integral, we have
  \begin{equation}
    \label{eq:Haar-integral-2}
    \Lambda_{(1)} \otimes a \Lambda_{(2)}
    = S(a) \Lambda_{(1)} \otimes \Lambda_{(2)}, \quad
    \Lambda_{(1)} a \otimes \Lambda_{(2)}
    = \Lambda_{(1)} \otimes \Lambda_{(2)} S(a)
  \end{equation}
  for all $a \in A$ \cite[Lemma~3.2]{MR1726707}. This implies that $E = S(\Lambda_{(1)}) \otimes \Lambda_{(2)}$ is a separability idempotent of $A$. Applying Theorem~\ref{thm:FS-thm-fd-C-st-alg} to $A_0^\tau$, we obtain
  \begin{equation*}
    \tilde{\sigma}_{\tau}(V) = \chi_V(S^{\tau}S(\Lambda_{(1)}) g \Lambda_{(2)}) = \chi_V(g \cdot \tau(\Lambda_{(1)}) \Lambda_{(2)}).
  \end{equation*}
  By~\eqref{eq:Haar-integral-2}, $\tau(\Lambda_{(1)}) \Lambda_{(2)}$ is a central element. Hence, by Remark~\ref{rem:character},
  \begin{equation}
    \label{eq:FS-thm-WHCst-2}
    \tilde{\sigma}_{\tau}(V)
    = \frac{\chi_V(g)}{\chi_V(1)} \chi_V(\tau(\Lambda_{(1)}) \Lambda_{(2)})
    = \sigma_{\tau}(V).
  \end{equation}
  Theorem~\ref{thm:FS-thm-WHCst} follows immediately from \eqref{eq:FS-thm-WHCst-2} with $\tau = \id_A$. For general cases, we need to interpret \eqref{eq:FS-thm-WHCst-2} in terms of the canonical real form $A_0$. For this purpose, let, in general, $F: A \to M_n(\mathbb{C})$ be a $\mathbb{C}$-linear map. Then:
  \begin{equation}
    \label{eq:FS-thm-WHCst-3}
    F(A_0^\tau) \subset M_n(\mathbb{R}) \iff F(\tau(z)) = \overline{F(z)} \text{\quad for all $z \in A_0$}.
  \end{equation}
  Indeed, suppose that $F(A_0^\tau) \subset M_n(\mathbb{R})$. For $a \in A$, we set $x = 2^{-1}(a+\tau(\overline{a}))$ and $y = (2\mathbf{i})^{-1} (a-\tau(\overline{a}))$, where $\overline{a} = S(a)^*$. It is easy to check that $x, y \in A_0^\tau$, $a = x + \mathbf{i} y$ and $\tau(\overline{a}) = x - \mathbf{i} y$. Hence, if $a \in A_0$, then:
  \begin{equation*}
    F(\tau(a)) = F(\tau(\overline{a}))
    = F(x) - \mathbf{i} F(y)
    = \overline{F(x) + \mathbf{i} F(y)}
    = \overline{F(a)}.
  \end{equation*}
  Conversely, suppose that $F(\tau(z)) = \overline{F(z)}$ for all $z \in A_0$. Let $a \in A$ and write it as $a = x + \mathbf{i} y$ ($x, y \in A_0$). If $a \in A_0$, then:
  \begin{align*}
    \overline{F(a)}
    = \overline{F(\tau(\overline{a}))}
    & = \overline{F(\tau(x)) - \mathbf{i} \cdot F(\tau(y))} \\
    & = \overline{F(\tau(x))} + \mathbf{i} \cdot \overline{F(\tau(y))}
    = F(x) + \mathbf{i} F(x)
    = F(a).
  \end{align*}
  Now (1) is proved as follows: By~\eqref{eq:FS-thm-WHCst-2}, $\sigma_{\tau}(V) = +1$ if and only if $V$ admits a basis such that the matrix representation $\rho: A \to M_n(\mathbb{C})$ with respect to that basis has the following property:
  \begin{equation}
    \label{eq:FS-thm-WHCst-4}
    \rho(A_0^\tau) \subset M_n(\mathbb{R}).
  \end{equation}
  By~\eqref{eq:FS-thm-WHCst-3} with $F = \rho$, we see that \eqref{eq:FS-thm-WHCst-4} is equivalent to that $\rho(\tau(a)) = \overline{\rho(a)}$ holds for all $a \in A_0$. Hence (1) follows. To show (2) and (3), use \eqref{eq:FS-thm-WHCst-3} with $F = \chi_V$.
\end{proof}

\subsection{Example II. Table algebras}

Let $A$ be a finite-dimensional $*$-algebra with basis $\mathcal{B} = \{ b_i \}_{i \in I}$, where $I$ is an index set. We suppose that $\mathcal{B}$ is closed under the $*$-operation; thus, for each $i \in I$, we can define $i^* \in I$ by $(b_i)^* = b_{i^*}$. For $i, j, k \in I$, we define $p_{i j}^k \in \mathbb{C}$ by
\begin{equation*}
  b_i \cdot b_j = \sum_{k \in I} p_{i j}^k b_k \quad (i, j \in I).
\end{equation*}
The pair $(A, \mathcal{B})$ is called a {\em table algebra} if the following conditions are satisfied:
\begin{enumerate}
\item[(T0)] $1_A \in \mathcal{B}$; the element $i \in I$ such that $b_i = 1_A$ will be denoted by $0$.
\item[(T1)] $p_{i j}^k \in \mathbb{R}$ for all $i, j, k \in I$.
\item[(T2)] $p_{i i^*}^0 = p_{i^* i}^0 > 0$ and $p_{i j}^0 = 0$ whenever $i \ne j^*$.
\end{enumerate}
Note that in literature the word `table algebra' has been used with several different meanings, see \cite{MR2535395}. Our definition is equivalent to \cite[Definition 1.16]{MR2535395}. One can find many examples of table algebras in \cite{MR2535395} and references therein. The Bose-Mesner algebra of an association scheme is an important example.

Now suppose that $(A, \mathcal{B})$ is a table algebra. By (T1), $A_0 = \mathrm{span}_{\mathbb{R}}(\mathcal{B})$ is a real form of $A$, which we call the {\em canonical real form} of $(A, \mathcal{B})$. Let $\phi: A \to \mathbb{C}$ be the linear map determined by $\phi(b_i) = \delta_{i 0}$. By (T2), the sesquilinear form
\begin{equation*}
  \langle a | b \rangle = \phi(a^* b) \quad (a, b \in A)
\end{equation*}
is an inner product on $A$ making it into a $*$-representation \cite[\S2]{MR2535395}. Note that $\mathcal{B}$ is orthogonal (but not normal in general) with respect this inner product. Hence, by the arguments of Remark \ref{rem:fd-C-st-sep-idemp},
\begin{equation}
  \label{eq:tbl-alg-sep-idemp}
  E = \sum_{i \in I} \frac{1}{p_{i i^*}^0} b_{i^*} \otimes b_i v^{-1},
  \quad \text{where }
  v = \sum_{i \in I} \frac{1}{p_{i i^*}^0} b_{i^*}  b_i,
\end{equation}
is a separability idempotent of $A$. Now we formulate {\em the Frobenius-Schur theorem for table algebras} as follows:

\begin{theorem}
  \label{thm:FS-thm-tbl-alg}
  For a simple left $A$-module $V$,
  \begin{equation}
    \label{eq:FS-thm-tbl-alg}
    \sum_{i \in I} \frac{1}{p_{i i^*}^0} \chi_V(b_i^2)
    = \frac{\chi_V(v)}{\chi_V(1)} \times \begin{cases}
      + 1 & \text{if $V$ is real}, \\
      \ 0 & \text{if $V$ is complex}, \\
      - 1 & \text{if $V$ is quaternionic}
    \end{cases}
  \end{equation}
  with respect to the canonical real form $A_0 = \mathrm{span}_{\mathbb{R}}(\mathcal{B})$.
\end{theorem}

As we have seen in Remark \ref{rem:fd-C-st-sep-idemp}, $\chi_V(v) > 0$. Hence, in particular, the left-hand side of \eqref{eq:FS-thm-tbl-alg} is positive, zero or negative according to whether $V$ is real, complex or quaternionic with respect to $A_0$.

Instead of proving Theorem~\ref{thm:FS-thm-tbl-alg}, we prove the following {\em twisted version}: By an {\em involution} of the table algebra $(A, \mathcal{B})$, we mean an automorphism $\tau$ of the $*$-algebra $A$ such that $\tau(\mathcal{B}) \subset \mathcal{B}$ and $\tau^2 = \id_A$. Now let $\tau$ be an involution of $(A, \mathcal{B})$. For each $i \in I$, we define $\tau(i) \in I$ by $b_{\tau(i)} = \tau(b_i)$.

\begin{theorem}
  For a simple left $A$-module $V$, we set
  \begin{equation*}
    \sigma_{\tau}(V) = \frac{\chi_V(1)}{\chi_V(v)}
    \sum_{i \in I} \frac{1}{p_{i i^*}^0} \chi_V(b_{\tau(i)} b_i).
  \end{equation*}
  Then $\sigma_{\tau}(V) \in \{ 0, \pm 1 \}$. Moreover, $\sigma_{\tau}(V)$ has the following meaning:
  \begin{enumerate}
  \item $\sigma_{\tau}(V) = +1$ if and only if $V$ admits a basis such that the matrix representation $\rho: A \to M_n(\mathbb{C})$ with respect to that basis fulfills:
    \begin{equation*}
      \rho(b_{\tau(i)}) = \overline{\rho(b_i)} \quad \text{for all $i \in I$}.
    \end{equation*}
  \item $\sigma_{\tau}(V) = -1$ if and only if $V$ does not admit such a basis but satisfies:
    \begin{equation*}
      \chi_V(b_{\tau(i)}) = \overline{\chi_V(b_i)} \quad \text{for all $i \in I$}.
    \end{equation*}
  \item $\sigma_{\tau}(V) = 0$ if and only if $\chi_V(b_{\tau(i)}) \ne \overline{\chi_V(b_i)}$ for some $i \in I$.
  \end{enumerate}
\end{theorem}

If $G$ is a finite group, then the pair $(\mathbb{C}G, G)$ is a table algebra. Applying this theorem, we again obtain the result of Kawanaka and Matsuyama \cite{MR1078503}. Thus this theorem is another generalization of their result.

\begin{proof}
  Define a linear map $S^{\tau}: A \to A$ by $S^{\tau}(b_i) = \tau(b_i)^*$ ($i \in I$). Then
  \begin{equation*}
    A_0^{\tau} := \{ a \in A \mid S^{\tau}(a)^* = a \}
  \end{equation*}
  is a real form of $A$. We define $\tilde{\sigma}_{\tau}(V)$ to be $+1$, $0$ or $-1$ according to whether $V$ is real, complex or quaternionic with respect to the real form $A_0^{\tau}$. Note that the element $v$ is central. By Theorem~\ref{thm:FS-thm-fd-C-st-alg} with $E$ given by~\eqref{eq:tbl-alg-sep-idemp}, we obtain
  \begin{equation*}
    \tilde{\sigma}_{\tau}(V)
    = \sum_{i \in I} \frac{1}{p_{i i^*}} \chi_V(b_{\tau(i)} b_i v^{-1})
    = \frac{\chi_V(1)}{\chi_V(v)}
    \sum_{i \in I} \frac{1}{p_{i i^*}} \chi_V(b_{\tau(i)} b_i)
    = \sigma_{\tau}(V).
  \end{equation*}
  Now (1)--(3) are proved by interpreting this result in terms of the canonical real form $A_0$ in a similar way as the proof of Theorem~\ref{thm:FS-thm-WHCst-tw}.
\end{proof}

\section{Compact quantum groups}
\label{sec:CQG}

\subsection{Conventions}

Given a coalgebra $C = (C, \Delta, \varepsilon)$ over $\mathbb{C}$, we denote by $\mathcal{M}^C_{\fd}$ the category of finite-dimensional right $C$-comodules. The coaction of $C$ on $V \in \mathcal{M}^C_{\fd}$ is expressed by the Sweedler notation:
\begin{equation*}
  V \to V \otimes C, \quad v \mapsto v_{(0)} \otimes v_{(1)} \quad (v \in V).
\end{equation*}
If a basis $\{ v_i \}_{i = 1}^n$ of $V$ is given, then we can define $(c_{i j}) \in M_n(C)$ by
\begin{equation}
  \label{eq:matrix-corep}
  v_{j(0)} \otimes v_{j(1)} = \sum_{i = 1}^n v_i \otimes c_{i j} \quad (j = 1, \dotsc, n).
\end{equation}
The matrix $(c_{i j})$ is called the {\em associated matrix corepresentation with respect to the basis $\{ v_i \}$}. The subspace spanned by $c_{i j}$'s will be denoted by $C(V)$:
\begin{equation*}
  C(V) := \mathrm{span}_{\mathbb{C}} \{ c_{i j} \mid i, j = 1, \dotsc, n \}.
\end{equation*}
Note that $C(V)$ does not depend on the choice of the basis $\{ v_i \}$. The trace of $(c_{i j})$, $t_V := c_{1 1} + \dotsb + c_{n n} \in C(V)$, also does not depend on the choice of $\{ v_i \}$. We call $t_V$ the {\em character} of $V$.

\subsection{Real forms of a coalgebra}

By a {\em real form} of a coalgebra $C$, we mean an $\mathbb{R}$-subspace $C_0 \subset C$ such that
\begin{equation}
  \label{eq:RF-cog-def}
  C = C_0 \oplus \mathbf{i} C_0, \quad \Delta(C_0) \subset \mathrm{span}_{\mathbb{R}} \{ c' \otimes c'' \mid c', c'' \in C_0 \}.
\end{equation}
Given a real form $C_0 \subset C$, we define $\overline{x + \mathbf{i} y} = x - \mathbf{i} y$ ($x, y \in C_0$). For $c \in C$, we call $\overline{c}$ the {\em conjugate} of $c$ with respect to the real form $C_0$. By~\eqref{eq:RF-cog-def}, taking the conjugate is an anti-linear operator on $C$ such that
\begin{equation}
  \label{eq:RF-cog-conj}
  \overline{\overline{c}} = c, \quad \Delta(\overline{c}) = \overline{c_{(1)}} \otimes \overline{c_{(2)}}.
\end{equation}
Conversely, if we are given an anti-linear operator $\overline{\phantom{a}}: C \to C$ with this property, then $\{ c \in C \mid \overline{c} = c \}$ is a real form of $C$. Thus giving a real form of $C$ is equivalent to giving an anti-linear operator on $C$ satisfying~\eqref{eq:RF-cog-conj}.

Fix a real form $C_0 \subset C$. If $V \in \mathcal{M}^C_{\fd}$, then $C$ also coacts on $\overline{V}$ by
\begin{equation*}
  \overline{V} \to \overline{V} \otimes_{\mathbb{C}} C,
  \quad \overline{v} \mapsto \overline{v_{(0)}} \otimes \overline{v_{(1)}}
  \quad (v \in V).
\end{equation*}
We call $\overline{V}$ the {\em conjugate} of $V$ with respect to $C_0$. In the same way as in \S\ref{subsec:RF-alg}, taking the conjugate defines a real structure for $\mathcal{M}^C_{\fd}$, which will be referred to as the {\em real structure for $\mathcal{M}^C_{\fd}$ associated with the real form $C_0$}. Following \S\ref{subsec:RF-alg}, we introduce the following definition:

\begin{definition}
  Let $C$ be a coalgebra with real form $C_0$, and let $\mathcal{J}$ be the real structure for $\mathcal{M}^C_{\fd}$ associated with $C_0$. We say that a simple comodule $V \in \mathcal{M}^C_{\fd}$ is {\em real}, {\em complex} and {\em quaternionic} with respect to $C_0$ if $\sigma_{\mathcal{J}}(V)$ is equal to $+1$, $0$ and $-1$, respectively.
\end{definition}

To characterize the $\mathcal{J}$-signature of a simple comodule, we recall that the dual space $A := C^{\ast}$ is an algebra over $\mathbb{C}$ with respect to the convolution product $\star$ defined by $\langle a \star b, c \rangle = \langle a, c_{(1)} \rangle \langle b, c_{(2)} \rangle$ ($a, b \in A$, $c \in C$). The algebra $C^*$ is called the {\em dual algebra} of $C$. There is the following relation between real forms of $C$ and those of $A$.

\begin{lemma}
  \label{lem:RF-dual-alg}
  If $C_0$ is a real form of a coalgebra $C$, then
  \begin{equation}
    \label{eq:RF-dual-alg-1}
    A_0 = \{ a \in A \mid \text{$a(c) \in \mathbb{R}$ for all $c \in C_0$} \}
  \end{equation}
  is a real form of $A = C^{\ast}$. The real form $C_0$ can be recovered from $A_0$ by
  \begin{equation}
    \label{eq:RF-dual-alg-2}
    C_0 = \{ c \in C \mid \text{$a(c) \in \mathbb{R}$ for all $a \in A_0$} \}.
  \end{equation}
\end{lemma}
\begin{proof}
  For $a \in A$, we define $\overline{a} \in C$ by $\overline{a}(c) = \overline{a(\overline{c})}$ ($c \in C$), where $\overline{c}$ is the conjugate of $c$ with respect to $C_0$. Then it is easy to check that
  \begin{equation*}
    A_0 = \{ a \in A \mid \overline{a} = a \}.
  \end{equation*}
  Since $\overline{\overline{a}} = a$ and $\overline{a \star b} = \overline{a} \star \overline{b}$ ($a, b \in A$), $A_0$ is a real form. To show \eqref{eq:RF-dual-alg-2}, we denote by $C_1$ the right-hand side of~\eqref{eq:RF-dual-alg-2}. $C_0 \subset C_1$ is trivial. Now let $c \in C_1$ and write it as $c = x + \mathbf{i} y$ ($x, y \in C_0$). Then, for all $a \in A_0$,
  \begin{equation*}
    a(x) \in \mathbb{R}, \quad
    a(c) = a(x) + \mathbf{i} a(y) \in \mathbb{R}.
  \end{equation*}
  This implies that $a(y) = 0$ for all $a \in A_0$. Since $A$ is spanned by $A_0$ over $\mathbb{C}$, we see that $a(y) = 0$ for all $a \in A$. Hence $y = 0$.
\end{proof}

Let $C$ be a coalgebra, and let $A = C^{\ast}$. Given $V \in \mathcal{M}^C_{\fd}$, we define $\Phi(V)$ to be the vector space $V$ endowed with the left $A$-action given by $a \cdot v = \langle a, v_{(1)} \rangle v_{(0)}$ for $a \in A$ and $v \in V$. $V \mapsto \Phi(V)$ defines a $\mathbb{C}$-linear functor
\begin{equation}
  \label{eq:Phi}
  \Phi: \mathcal{M}^C_{\fd} \to {{}_A\mathcal{M}_{\fd}}, \quad V \mapsto \Phi(V),
\end{equation}
which is well-known to be fully faithful (see, {\it e.g.}, \cite[Chapter 2]{MR1786197}). Now let $C_0$ be a real form of $C$, and let $A_0$ be the real form of $A$ given by \eqref{eq:RF-dual-alg-2}. $C_0$ and $A_0$ define real structures for $\mathcal{M}^C_{\fd}$ and ${{}_A\mathcal{M}_{\fd}}$, respectively. Abusing notation, we denote them by the same symbol $\mathcal{J} = (J, i)$. It is easy to check that $\Phi \circ J = J \circ \Phi$ and $\Phi(i) = i$. Hence, in particular,
\begin{equation}
  \label{eq:Phi-preserves-J-sign}
  \sigma_{\mathcal{J}}(V) = \sigma_{\mathcal{J}}(\Phi(V))
\end{equation}
for all simple comodule $V \in \mathcal{M}^C_{\fd}$. Now we give the following characterization of the $\mathcal{J}$-signature:

\begin{theorem}
  \label{thm:RF-cog-sign-1}
  Let $C$ be a coalgebra with real form $C_0$, and let $V \in \mathcal{M}^C$ be a simple comodule with character $t_V$. Then:
  \begin{enumerate}
  \item $V$ is real with respect to $C_0$ if and only if $V$ admits a basis $\{ v_i \}_{i = 1}^n$ such that the matrix corepresentation $(c_{i j})$ with respect to $\{ v_i \}$ fulfills:
    \begin{equation*}
      c_{i j} \in C_0 \text{\quad for all $i, j = 1, \dotsc, n$}.
    \end{equation*}
  \item $V$ is quaternionic with respect to $C_0$ if and only if $V$ does not admit such a basis but $t_V \in C_0$.
  \item $V$ is complex with respect to $C_0$ if and only if $t_V \not \in C_0$.
  \end{enumerate}
\end{theorem}
\begin{proof}
  Let $A_0$ be the real form of $A := C^{\ast}$ given by~\eqref{eq:RF-dual-alg-1}. To prove the claim, let, in general, $X \in \mathcal{M}^{C}_{\fd}$. Fix a basis $\{ v_i \}_{i = 1}^n$ of $X$ and define $c_{i j}$'s by~\eqref{eq:matrix-corep}. Then the action of $a \in A$ on $\Phi(X)$ is given by
  \begin{equation}
    \label{eq:dual-alg-mat-rep}
    a \cdot v_j = a(v_{j(1)}) v_{j(0)} = \sum_{i = 1}^n a(c_{i j}) v_i \quad (j = 1, \dotsc, n).
  \end{equation}
  Let $\rho: A \to M_n(\mathbb{C})$ be the matrix representation with respect to the basis $\{ v_i \}$ of the left $A$-module $\Phi(X)$. By Lemma~\ref{lem:RF-dual-alg}, we have:
  \begin{equation}
    \label{eq:RCH-cog-1}
    \rho(A_0) \subset M_n(\mathbb{R}) \iff \text{$c_{i j} \in C_0$ for all $i, j = 1, \dotsc, n$}.
  \end{equation}
  Let $t$ and $\chi$ be the characters of $X$ and $\Phi(X)$, respectively. By~\eqref{eq:dual-alg-mat-rep},
  \begin{equation}
    \label{eq:dual-alg-character}
    a(t) = a(c_{1 1}) + \dotsb + a(c_{n n}) = \chi(a) \quad (a \in A).
  \end{equation}
  Hence, again by Lemma~\ref{lem:RF-dual-alg}, we have:
  \begin{equation}
    \label{eq:RCH-cog-2}
    \chi(A_0) \subset \mathbb{R} \iff t \in C_0.
  \end{equation}
  By~\eqref{eq:Phi-preserves-J-sign}, $V$ is real, complex and quaternionic with respect to $C_0$ if and only if $\Phi(V)$ is real, complex and quaternionic with respect to $A_0$, respectively. The proof is done by interpreting Theorem~\ref{thm:RF-sign-1} in terms of $C$ by~\eqref{eq:RCH-cog-1} and~\eqref{eq:RCH-cog-2}.
\end{proof}

\subsection{Real forms of a $*$-coalgebra}
\label{subsec:RF-st-cog}

A {\em $*$-coalgebra} is a coalgebra $C$ endowed with an anti-linear operator $*: C \to C$ such that
\begin{equation*}
  c^{**} = c, \quad (c^*)_{(1)} \otimes (c^*)_{(2)} = (c_{(2)})^* \otimes (c_{(1)})^* \quad (c \in C).
\end{equation*}
Given a real form $C_0$ of a $*$-coalgebra $C$, we define
\begin{equation}
  \label{eq:RF-st-cog-1}
  \varsigma: C \to C, \quad \varsigma(c) = \overline{c^*} \quad (c \in C),
\end{equation}
where $\overline{\phantom{a}}$ is the conjugate with respect to $C_0$. We refer the linear map $\varsigma$ as the {\em anti-coalgebra map associated with $C_0$} since, by \eqref{eq:RF-cog-conj}, it satisfies
\begin{equation}
  \label{eq:RF-st-cog-2}
  \Delta(\varsigma(c)) = \varsigma(c_{(2)}) \otimes \varsigma(c_{(1)})
  \quad \varsigma(\varsigma(c^*)^*) = c
  \quad (c \in C).
\end{equation}
Conversely, if a linear map $\varsigma: C \to C$ satisfies~\eqref{eq:RF-st-cog-1}, then
\begin{equation}
  \label{eq:RF-st-cog-3}
  C_0 = \{ c \in C \mid \varsigma(c^*) = c \}
\end{equation}
is a real form of $C$. We call $C_0$ the {\em real form associated with $\varsigma$}. It is easy to check that \eqref{eq:RF-st-cog-1} and~\eqref{eq:RF-st-cog-3} establish a bijection between real forms of $C$ and linear maps satisfying~\eqref{eq:RF-st-cog-2}.

Note that the dual algebra $A = C^{\ast}$ is a $*$-algebra by
\begin{equation*}
  \langle a^*, c \rangle = \overline{\langle a, c^* \rangle} \quad (a \in A, c \in C).
\end{equation*}
If $\varsigma: C \to C$ is a linear map satisfying~\eqref{eq:RF-st-cog-1}, then
\begin{equation}
  \label{eq:RF-dual-st-alg}
  S: A \to A, \quad a \mapsto a \circ \varsigma \quad (a \in A)
\end{equation}
is a linear map satisfying~\eqref{eq:RF-st-alg-2}. The above constructions for the $*$-coalgebra $C$ are summarized in the following commutative diagram:
\begin{equation*}
  \begin{CD}
    \{ \text{real forms of $C$} \}
    @>{\eqref{eq:RF-st-cog-1}}>{}>
    \{ \text{$\varsigma \in \End_{\mathbb{C}}(C)$ satisfying \eqref{eq:RF-st-cog-2}} \} \\
    @V{\eqref{eq:RF-dual-alg-1}}V{}V @V{}V{\eqref{eq:RF-dual-st-alg}}V \\
    \{ \text{real forms of $A = C^{\ast}$} \}
    @>{}>{\eqref{eq:RF-st-alg-1}}>
    \{ \text{$S \in \End_{\mathbb{C}}(A)$ satisfying \eqref{eq:RF-st-alg-2}} \}.
  \end{CD}
\end{equation*}

By a {\em $*$-corepresentation} of $C$, we mean a finite-dimensional Hilbert space $V$ endowed with a right $C$-comodule structure such that
\begin{equation*}
  \langle v_{(0)} | w \rangle (v_{(1)})^* = \langle v | w_{(0)} \rangle w_{(1)}
  \quad (v, w \in V).
\end{equation*}
We denote by $\fdCorep(C)$ the category whose objects are $*$-corepresentations of $C$ and whose morphisms are maps of right $C$-comodules.

\begin{proposition}
  $\fdCorep(C)$ is a $C^*$-category.
\end{proposition}

Now let $C_0$ be a real form of $C$, and let $\varsigma: C \to C$ be the anti-coalgebra map associated with $C_0$. Suppose moreover that there exists an invertible positive element $\gamma$ of the dual $*$-algebra $A = C^*$ such that
\begin{equation*}
  \gamma \circ \varsigma = \gamma^{-1}, \quad
  \varsigma^2(c) = \gamma(c_{(1)}) c_{(2)} \gamma^{-1}(c_{(3)}) \quad (c \in C).
\end{equation*}
Given $X \in \fdCorep(C)$, we define $D(X)$ to be the dual space $X^\vee$ endowed with the right $C$-comodule structure determined by
\begin{equation*}
  \langle \lambda_{(0)}, x \rangle \lambda_{(1)}
  = \langle \lambda, x_{(0)} \rangle \varsigma(x_{(1)})
  \quad (\lambda \in X^\vee, x \in X)
\end{equation*}
and the inner product given by
\begin{equation*}
  \langle \phi_x | \phi_y \rangle_{D(X)}
  = \langle y | x_{(0)} \rangle \gamma(x_{(1)})
  \quad (x, y \in X).
\end{equation*}
We also define $J(X)$ to be the right $C$-comodule $\overline{X}$ ({\it i.e.}, the conjugate comodule with respect to the real form $C_0$) endowed with the inner product given by
\begin{equation*}
  \langle \overline{x} | \overline{y} \rangle_{J(X)}
  = \langle y | x_{(0)} \rangle \gamma(x_{(1)})
  \quad (x, y \in X).
\end{equation*}
$D(X)$ and $J(X)$ are $*$-corepresentations. Moreover, $D$ and $J$ are (contravariant) endo-$*$-functors on $\fdCorep(C)$. We now define natural isomorphisms $i: \id \to J J$ and $\eta: \id \to D D$ by \eqref{eq:RF-conj-rep-4} and
\begin{equation*}
  \langle \eta_X(x), f \rangle = \langle f, x_{(0)} \rangle \gamma(x_{(1)})
  \quad (X \in \fdCorep(C), x \in X, f \in X^\vee),
\end{equation*}
respectively. Then:

\begin{lemma}
  $\mathcal{D} = (D, \eta)$ and $\mathcal{J} = (J, i)$ are a $*$-compatible dual structure and a $*$-compatible real structure for $\fdCorep(C)$, respectively. Moreover, the dual structures $\mathbb{D}(\mathcal{J})$ and $\mathcal{D}$ are unitary equivalent. Hence, by Theorem~\ref{thm:FS-thm-C-st}, we obtain
  \begin{equation}
    \label{eq:FS-thm-star-coalg}
    \nu_{\mathcal{D}}(V) = \sigma_{\mathcal{J}}(V).
  \end{equation}
  for all irreducible $*$-corepresentation $V \in \fdCorep(C)$.
\end{lemma}

To verify this lemma, observe that if $X \in \fdCorep(C)$, then $\Phi(X)$ is a $*$-\hspace{0pt}representation of the dual $*$-algebra $A$ with respect to the same inner product. Hence \eqref{eq:Phi} induces a fully faithful $*$-functor
\begin{equation}
  \label{eq:Phi-rep}
  \Phi: \fdCorep(C) \to \Rep(A).
\end{equation}
Let $A_0$ be the real form of $A$ given by \eqref{eq:RF-dual-alg-1}, let $S$ be the anti-algebra map associated with $A_0$, and let $g = \gamma$. By the same way as in \S\ref{subsec:RF-st-alg}, $A_0$, $S$ and $g$ define a real structure and a dual structure for $\Rep(A)$. Abusing notation, we denote them by $\mathcal{J} = (J, i)$ and $\mathcal{D} = (D, \eta)$, respectively. Then one easily checks that
\begin{equation*}
  \Phi \circ J = J \circ \Phi, \quad \Phi(i) = i, \quad \Phi \circ D = D \circ \Phi, \quad \Phi(\eta) = \eta.
\end{equation*}
Now the lemma follows from the results of \S~\ref{subsec:RF-st-alg} by regarding $\fdCorep(C)$ as a full subcategory of $\Rep(A)$ via \eqref{eq:Phi-rep}.

\subsection{Frobenius-Schur theorem for compact coalgebras}

A {\em coseparability idempotent} of a coalgebra $C$ is a bilinear map $E: C \times C \to \mathbb{C}$ such that
\begin{equation*}
  E(c_{(1)}, c_{(2)}) = \varepsilon(c), \quad
  c_{(1)} \cdot E(c_{(2)}, d) = E(c, d_{(1)}) \cdot d_{(2)}
  \quad (c, d \in C).
\end{equation*}
If $C$ is finite-dimensional, then a coseparability idempotent of $C$ is nothing but a separability idempotent of the dual algebra $A = C^*$.

Let $M_n^c(\mathbb{C})$ be the matrix $*$-coalgebra of degree $n$, {\it i.e.}, the $\mathbb{C}$-vector space with basis $\{ e_{i j} \}_{i, j = 1, \dotsc, n}$ endowed with the $*$-coalgebra structure determined by
\begin{equation*}
  \Delta(e_{i j}) = \sum_{k = 1}^n e_{i k} \otimes e_{k j},
  \quad \varepsilon(e_{i j}) = \delta_{i j},
  \quad e_{i j}^* = e_{j i}
  \quad (i, j = 1, \dotsc, n).
\end{equation*}
Following \cite[\S2]{MR1382806}, we say that a $*$-coalgebra $C$ is {\em compact} if $C = \sum_{X} C(X)$, where $X$ runs over all $*$-corepresentations of $C$. This class of $*$-coalgebras is characterized as follows (see also \cite[Lemma~2.1]{MR1382806}):

\begin{proposition}
  \label{prop:compact-coalg}
  For a $*$-coalgebra $C$, the following are equivalent:
  \begin{enumerate}
  \item $C$ is compact.
  \item $C = \sum_V C(V)$, where $V$ runs over all irreducible $*$-corepresentations.
  \item $C$ is isomorphic to a direct sum of matrix $*$-coalgebras.
  \item $C$ has a coseparability idempotent $E: C \times C \to \mathbb{C}$ such that
    \begin{equation*}
      E(c^*, d^*) = \overline{E(d, c)}, \quad E(c^*, c) > 0
      \quad (c, d \in C, c \ne 0).
    \end{equation*}
  \item Any finite-dimensional right $C$-comodule admits an inner product making it into a $*$-corepresentation of $C$.
  \end{enumerate}
\end{proposition}
\begin{proof}
  Let $\{ V_{\lambda} \}_{\lambda \in \Lambda}$ be a complete set of representatives of isomorphism classes of irreducible $*$-corepresentations of $C$. Note that, by the fundamental theorem of coalgebras, all $V_{\lambda}$'s are finite-dimensional. By considering the associated matrix corepresentation with respect to an orthonormal basis of $V_{\lambda}$, we see that $C(V_{\lambda})$ is a $*$-subcoalgebra of $C$ isomorphic such that
  \begin{equation*}
    C(V_{\lambda}) \cong M_{n_{\lambda}}(\mathbb{C}),
    \text{\quad where $n_{\lambda} := \dim_{\mathbb{C}}(V_{\lambda})$}.
  \end{equation*}
  \noindent $(1) \Rightarrow (2)$. Let, in general, $X \in \fdCorep(C)$. If $V \subset X$ is a subcomodule, then its orthogonal complement $V^{\perp}$ is also a subcomodule and $X = V \oplus V^{\perp}$. This shows that $X$ is isomorphic to a direct sum of $V_{\lambda}$'s. Hence $C(X) \subset \sum_{\lambda \in \Lambda} C(V_{\lambda})$. By the definition of a compact $*$-coalgebra,
  \begin{equation*}
    C = \sum_{X} C(X) \subset \sum_{\lambda \in \Lambda} C(V_{\lambda}) \subset C,
  \end{equation*}
  where the first sum is taken over all $X \in \fdCorep(C)$.

  \vspace{.5\baselineskip}
  \noindent $(2) \Rightarrow (3)$. By the assumption, $C = \sum_{\lambda \in \Lambda} C(V_{\lambda})$. Since $V_{\lambda}$'s are simple and mutually non-isomorphic, the sum is in fact a direct sum. Hence we obtain an isomorphism of $*$-coalgebras $C = \bigoplus_{\lambda \in \Lambda} C(V_{\lambda}) \cong \bigoplus_{\lambda \in \Lambda} M_{n_{\lambda}}^c(\mathbb{C})$.

  \vspace{.5\baselineskip}
  \noindent $(3) \Rightarrow (4)$. Fix an isomorphism $C \cong \bigoplus_{a \in A} M_{n_{a}}^c(\mathbb{C})$ of $*$-coalgebras. Let $e_{i j}^{(a)} \in C$ be the element corresponding to the element $e_{i j}$ of the $a$-th component $M_{n_{a}}^c(\mathbb{C})$ via the isomorphism. Then the bilinear map $E: C \times C \to \mathbb{C}$ determined by
  \begin{equation*}
    E(e_{i j}^{(a)}, e_{k l}^{(b)}) = \delta_{a b} \delta_{i l} \delta_{j k}
    \quad (a, b \in A; i, j = 1, \dotsc, n_{a}; k, l = 1, \dotsc, n_{b})
  \end{equation*}
  is a coseparability idempotent with the desired properties.

  \vspace{.5\baselineskip}
  \noindent $(4) \Rightarrow (5)$. Let $E$ be such a coseparability idempotent of $C$. Given $X \in \mathcal{M}_{\fd}^C$, we fix an inner product $\langle|\rangle_0$ on $X$ and then define
  \begin{equation*}
    \langle x | y \rangle = \langle x_{(0)} | y_{(0)} \rangle_0 \cdot E(x_{(1)}^*, y_{(1)})
    \quad (x, y \in X).
  \end{equation*}
  It is straightforward to verify~\eqref{eq:FS-thm-star-coalg}. By $E(c^*, d^*) = \overline{E(d, c)}$, $\langle | \rangle$ is Hermitian. To show that $\langle | \rangle$ is positive definite, fix an orthonormal basis $\{ e_i \}_{i = 1}^n$ of $X$ with respect to $\langle | \rangle_0$ and let $(c_{i j})$ be the matrix corepresentation with respect to $\{ e_i \}$. Now let $j \in \{ 1, \dotsc, n \}$. By using the counit, we see that $c_{i j} \ne 0$ for some $i$. Hence:
  \begin{equation*}
    \langle e_j | e_j \rangle
    = \sum_{i, k = 1}^n \langle e_i | e_k \rangle E(c_{i j}^*, c_{k j})
    = \sum_{i = 1}^n E(c_{i j}^*, c_{i j}) > 0.
  \end{equation*}

  \noindent $(5) \Rightarrow (1)$. The fundamental theorem of coalgebras asserts that $C = \sum_X C(X)$, where the sum is taken over all $X \in \mathcal{M}^C_{\fd}$. By the assumption, we may assume that all $X$'s are $*$-corepresentation. Thus $C$ is compact.
\end{proof}

By Remark~\ref{rem:fd-C-st-characterization} and the above proposition, a finite-dimensional $*$-coalgebra is compact if and only if its dual $*$-algebra is a $C^*$-algebra. Note that a compact $*$-colagebra is not necessarily finite-dimensional while an algebra over $\mathbb{C}$ having a separability idempotent is always finite-dimensional.

The following proposition can be considered as the dual of Proposition~\ref{prop:BNS}:

\begin{proposition}
  \label{prop:BNS-cog}
  Suppose that a $*$-coalgebra $C$ is compact. Then, for every linear map $\varsigma: C \to C$ satisfying~\eqref{eq:RF-st-cog-2}, there uniquely exists an invertible positive element $\gamma \in C^{\ast}$ satisfying the following conditions:
  \begin{enumerate}
  \item[(1)] $\varsigma^2(c) = \gamma(c_{(1)}) c_{(2)} \gamma^{-1}(c_{(3)})$ for all $c \in C$.
  \item[(2)] $\gamma(t) = \gamma^{-1}(t)$ if $t$ is the character of some simple right $C$-comodule.
  \end{enumerate}
  Moreover, such an element $\gamma$ satisfies the condition
  \begin{enumerate}
  \item[(3)] $\gamma \circ \varsigma = \gamma^{-1}$.
  \end{enumerate}
\end{proposition}

We omit the proof here; see Appendix~A. Now we prove:

\begin{theorem}
  \label{thm:FS-thm-cpt-cog}
  Let $C$ be a compact coalgebra with real form $C_0$, and let $E$ be a coseparability idempotent for $C$. Then, for all simple comodule $V \in \mathcal{M}_{\fd}^C$,
  \begin{equation*}
    \gamma(t_{V(2)}) E(\varsigma(t_{V(1)}), t_{V(3)})
    = \begin{cases}
      + 1 & \text{if $V$ is real}, \\
      \ 0 & \text{if $V$ is complex}, \\
      - 1 & \text{if $V$ is quaternionic}
    \end{cases}
  \end{equation*}
  with respect to $C_0$, where $\varsigma: C \to C$ is the anti-coalgebra map associated with $C_0$ and $\gamma \in C^{\ast}$ is the element of Proposition \ref{prop:BNS-cog}.
\end{theorem}
\begin{proof}
  Define $\mathcal{J}$ and $\mathcal{D}$ as in the previous subsection. By definition, $\sigma_{\mathcal{J}}(V)$ is equal to $+1$, $0$ and $-1$ according to whether $V$ is real, complex or quaternionic with respect to $C_0$. On the other hand,
  \begin{equation*}
    \sigma_{\mathcal{J}}(V) = \nu_{\mathcal{D}}(V) = \gamma(t_{V(2)}) E(\varsigma(t_{V(1)}), t_{V(3)})
  \end{equation*}
  by \eqref{eq:FS-thm-star-coalg} and the coalgebraic version of Theorem~\ref{thm:FS-ind-sep-alg} (see \cite{KenichiShimizu:2012-11} for details).
\end{proof}

\subsection{Compact quantum groups}

We apply our results to {\em compact quantum groups} proposed by Woronowicz \cite{MR1616348}. Instead of his original $C^*$-algebraic approach, we adopt the approach of Koornwinder and Dijkhuizen \cite{MR1431306,MR1310296}. See \cite[\S5]{MR1310296} for the comparison between these approaches.

Following \cite[Definition~2.2]{MR1310296}, a {\em compact quantum group algebra} (CQG algebra, for short) is a Hopf $*$-algebra $A$ spanned by the matrix elements of the unitary corepresentations of $A$. Here, we should recall that a {\em unitary corepresentation} of $A$ is a finite-dimensional Hilbert space $V$ endowed with a right $A$-comodule structure such that $\langle v_{(0)} | w \rangle (v_{(1)})^* = \langle v | w_{(0)} \rangle S(w_{(1)})$ for all $v, w \in V$, where $S$ is the antipode of $A$.

Now let $A$ be a CQG algebra. Note that a Hopf $*$-algebra is {\em not} a $*$-coalgebra in the sense of~\ref{subsec:RF-st-cog}. Thus we define an anti-linear operator $\dagger$ by
\begin{equation*}
  a^\dagger = S(a)^* \quad (a \in A).
\end{equation*}
Since $S$ is an anti-coalgebra map, $A$ is a $\dagger$-coalgebra, {\it i.e.}, a $*$-coalgebra with respect to the $*$-structure $\dagger: A \to A$. A ``$\dagger$-corepresentation'' of $A$ is nothing but a unitary corepresentation of $A$. Following, appropriate definitions of real, complex and quaternionic corepresentations of $A$ are:

\begin{definition}
  \label{def:RCH-CQG}
  Let $V$ be an irreducible corepresentation of the CQG algebra $A$ with character $t_V$.
  \begin{enumerate}
  \item $V$ is {\em real} if it admits a basis $\{ v_i \}_{i = 1}^n$ such that the matrix corepresentation $(c_{i j})$ with respect to $\{ v_i \}$ has the following property:
    \begin{equation*}
      c_{i j}^* = c_{i j}.
    \end{equation*}
  \item $V$ is {\em quaternionic} if it does not admit such a basis but $t_V^* = t_V$.
  \item $V$ is {\em complex} if $t_V^* \ne t_V$.
  \end{enumerate}
\end{definition}

By Theorem~\ref{thm:RF-cog-sign-1}, $V$ is real, complex and quaternionic if and only if it is real, complex and quaternionic with respect to the real form
\begin{equation*}
  A_0
  = \{ a \in A \mid S(a^\dagger) = a \}
  \quad (= \{ a \in A \mid a^* = a \})
\end{equation*}
of the $\dagger$-coalgebra $A$.

By Proposition~\ref{prop:compact-coalg}, the $\dagger$-coalgebra $A$ is compact. Since $S \circ \dagger \circ S \circ \dagger = \id_A$, there uniquely exists a `$\dagger$-positive' element $\gamma \in A^*$ satisfying the conditions (1)--(4) of Proposition~\ref{prop:BNS-cog}. Here, we are saying that $f \in A^*$ is $\dagger$-positive if $f = g^\dagger \star g$ for some $g \in A^*$, where $g^{\dagger}(a) = \overline{g(a^{\dagger})}$ ($a \in A$).

A {\em Haar functional} on $A$ is a linear map $h: A \to \mathbb{C}$ satisfying
\begin{equation*}
  h(1_A) = 1,
  \quad h(a_{(1)}) \cdot a_{(2)} = h(a) 1_A = h(a_{(2)}) \cdot a_{(1)}
  \quad (a \in A).
\end{equation*}
See \cite[\S3]{MR1310296} for the existence and the uniqueness of a Haar functional. Now let $h$ be the Haar functional on $A$. It easily follows that
\begin{equation}
  \label{eq:Haar-functional-2}
  h(a_{(1)} b) \cdot a_{(2)} = h(a b_{(1)}) \cdot S(b_{(2)}), \quad
  a_{(1)} \cdot h(a_{(2)} b) = S(b_{(1)}) \cdot h(a b_{(2)})
\end{equation}
for all $a, b \in A$. The {\em Frobenius-Schur theorem for compact quantum groups} is now stated as follows:

\begin{theorem}
  \label{thm:FS-thm-CQG}
  If $V$ is an irreducible corepresentation of $A$, then
  \begin{equation}
    \label{eq:FS-thm-QCG}
    h(t_{V(1)} t_{V(2)}) = \frac{\varepsilon(t_V)}{\gamma(t_V)} \times \begin{cases}
      + 1 & \text{if $V$ is real}, \\
      \ 0 & \text{if $V$ is complex}, \\
      - 1 & \text{if $V$ is quaternionic}
    \end{cases}
  \end{equation}
  in the sense of Definition~\ref{def:RCH-CQG}.
\end{theorem}

Since $\varepsilon(t_V) = \dim_{\mathbb{C}}(V) > 0$ and $\gamma(t_V) > 0$, this theorem implies that the left-hand side of \eqref{eq:FS-thm-QCG} is positive, zero and negative according to whether $V$ is real, complex or quaternionic.

\begin{proof}
  Define $\sigma(V)$ to be $+1$, $0$, $-1$ according to whether $V$ is real, complex or quaternionic. By~\eqref{eq:Haar-functional-2},
  \begin{equation*}
    E: A \times A \to \mathbb{C}, \quad E(a, b) = h(S(a) b) \quad (a, b \in A)
  \end{equation*}
  is a coseparability idempotent of the coalgebra $A$. By Theorem~\ref{thm:FS-thm-cpt-cog},
  \begin{equation*}
    \sigma(V) = \gamma(t_{V(2)}) E(S(t_{V(1)}), t_{V(3)})
    = \gamma(t_{V(1)}) h(t_{V(2)} t_{V(3)})
    = (\gamma \star h^{[2]}) (t_V),
  \end{equation*}
  where $h^{[2]}(a) = h(a_{(1)} a_{(2)})$ ($a \in A$). By \eqref{eq:Haar-functional-2}, $h^{[2]}$ is a central element of the dual algebra $A^*$. Now let $\chi$ be the character of the left $A$-module $\Phi(V)$. By Remark~\ref{rem:character} and \eqref{eq:dual-alg-character}, we compute:
  \begin{equation*}
    \sigma(V)
    = \chi(\gamma \star h^{[2]})
    = \frac{\chi(\gamma) \chi(h^{[2]})}{\chi(\varepsilon)}
    = \frac{\gamma(t_{V})}{\varepsilon(t_V)} \cdot h(t_{V(1)} t_{V(2)}).
    \qedhere
  \end{equation*}
\end{proof}

\begin{remark}
  If $A = \mathcal{O}(G)$ is the Hopf $*$-algebra of continuous representative functions on a compact group $G$, then $\gamma = \varepsilon$ (since $\mathcal{O}(G)$ is commutative) and the Haar functional $h$ on $\mathcal{O}(G)$ is given by
  \begin{equation*}
    h(f) = \int_G a(g) d \mu(g) \quad (a \in \mathcal{O}(G)).
  \end{equation*}
  The real form $A_0 = \{ a \in A \mid S(a^\dagger) = a \}$ is precisely the subset of $\mathcal{O}(G)$ consisting of $\mathbb{R}$-valued functions. Thus Definition~\ref{def:RCH-CQG} agrees with the ordinary definitions of real, complex and quaternionic representations. See \cite[Example~4.9]{KenichiShimizu:2012-11} for details on why the right-hand side of~\eqref{eq:FS-thm-QCG} is equal to that of \eqref{eq:intro-FS-ind}.
\end{remark}

\begin{remark}
  Woronowicz \cite{MR1616348} showed that there exists a unique family $\{ f_z \}_{z \in \mathbb{C}}$ of algebra maps $f_z: A \to \mathbb{C}$ characterized by numerous properties, including:
  \begin{gather*}
    f_0 = \varepsilon,
    \quad f_w \star f_z = f_{z + w},
    \quad f_{\overline{z}}(a^*) = \overline{f_{-z}(a)},
    \quad f_z(S(a)) = f_{-z}(a), \\
    \quad S^2(a) = f_1(a_{(1)}) a_{(2)} f_{-1}(a_{(3)})
    \quad (w, z \in \mathbb{C}, a \in A),
  \end{gather*}
  see also \cite{MR1431306}. $f_1$ is $\dagger$-positive, since $f_1 = f_{1/2}^{\dagger} \star f_{1/2}$. Going back to the construction of $\{ f_z \}$, we see that $f_1$ satisfies the conditions (1) and (2) of Proposition~\ref{prop:BNS-cog}. Hence, by the uniqueness, $\gamma = f_1$.
\end{remark}

\appendix

\section{Duality lifting problems}

\subsection{A lift yields a dual structure}

Let $A$ be a $*$-algebra with real form $A_0$. By using the anti-algebra map $S$ associated with $A_0$, we define a contravariant $\mathbb{C}$-linear functor $D_0: \Rep(A) \to \Mod{A}$ as in \S\ref{subsec:RF-st-alg}.

Let $\mathcal{C}$ be a full subcategory of $\Rep(A)$. Recall that a lift of $D_0$ on $\mathcal{C}$ is a contravariant endo-$*$-functor $D: \mathcal{C} \to \mathcal{C}$ such that $U \circ D = D_0$, where $U: \mathcal{C} \to \Mod{A}$ is the forgetful functor.

\begin{proposition}
  \label{prop:D0-lift}
  If there exists a lift $D$ of $D_0$ on $\mathcal{C}$, then there exists a natural isomorphism $\eta: \id_{\mathcal{C}} \to D D$ such that the pair $(D, \eta)$ is a $*$-compatible dual structure for $\mathcal{C}$.
\end{proposition}
\begin{proof}
  For each $X \in \mathcal{C}$, we define $\tilde{\eta}_X = \phi_{D(X)} \circ \phi_X: X \to D D(X)$ ({\it cf}. \eqref{eq:Riesz-bidual}). One can show that $\tilde{\eta} = \{ \tilde{\eta}_X \}_{X \in \mathcal{C}}$ defines a natural isomorphism $\tilde{\eta}: \id_{\mathcal{C}} \to D D$. Now, for each $X \in \mathcal{C}$, we define
  \begin{equation*}
    \eta_X = \tilde{\eta}_X \circ | \tilde{\eta}_X |^{-1}: X \to D D(X).
  \end{equation*}
  By Lemmas~\ref{lem:C-star-cat-1} and~\ref{lem:C-star-cat-2}, $\eta = \{ \eta_X \}_{X \in \mathcal{C}}$ is a unitary natural isomorphism. We show that the pair $(D, \eta)$ is indeed a $*$-compatible dual structure for $\mathcal{C}$. For this purpose, it is sufficient to verify~\eqref{eq:dual-str-1}.

  For this purpose, we set $\alpha_X = D(\tilde{\eta}_X) \circ \tilde{\eta}_{D(X)}$ for $X \in \mathcal{C}$. By the naturality of $\tilde{\eta}$, we have $\tilde{\eta}_{D D(X)} \circ \tilde{\eta}_X = D D(\tilde{\eta}_{X}) \circ \tilde{\eta}_X$. Since $\tilde{\eta}_X$ is an isomorphism, $\tilde{\eta}_{D D(X)} = D D(\tilde{\eta}_{X})$. Hence,
  \begin{equation}
    \label{eq:D0-lift-2}
    D(\alpha_{X}) = D(\tilde{\eta}_{D(X)}) D D(\tilde{\eta}_X) = D(\tilde{\eta}_{D(X)}) \tilde{\eta}_{D D(X)} = \alpha_{D(X)}.
  \end{equation}
  Let $\phi_X^*: D(X) \to X$ be the adjoint operator of $\phi_X: X \to D(X)$. For $x, y \in X$,
  \begin{align*}
    \langle (\phi_X^{*})^{\vee}(\phi_x), \phi_y \rangle
    = \langle \phi_x, \phi_X^*(\phi_y) \rangle
    = \langle x | \phi_X^*(\phi_y) \rangle_{X}
    = \langle \phi_x | \phi_y \rangle_{D(X)}.
  \end{align*}
  This implies $(\phi_X^*)^\vee = \phi_{D(X)}$. By using this result, we get:
  \begin{equation}
    \label{eq:D0-lift-3}
    D(|\tilde{\eta}_X|^2)
    = D(\tilde{\eta}_X) \circ (\phi_{D(X)}^*)^\vee \circ (\phi_{X}^*)^\vee
    = D(\tilde{\eta}_X) \tilde{\eta}_X
    = \alpha_X.
  \end{equation}
  Hence, by~\eqref{eq:D0-lift-2} and~\eqref{eq:D0-lift-3},
  \begin{equation*}
    D^2(|\tilde{\eta}_X|^2) = D^2(\tilde{\eta}_X^* \tilde{\eta}_X)
    = D(\alpha_X) = \alpha_{D(X)}
    = D(\tilde{\eta}_{D(X)}^* \tilde{\eta}_{D(X)})
    = D(|\tilde{\eta}_{D(X)}|^2).
  \end{equation*}
  Since $D$ is an anti-equivalence, $D(|\eta_X|^2) = |\tilde{\eta}_{D(X)}|^2$. Taking the square root, we obtain $D(|\tilde{\eta}_X|) = |\tilde{\eta}_{D(X)}|$. Now~\eqref{eq:dual-str-1} is proved as follows:
  \begin{align*}
    D(\eta_X) \eta_{D(X)}
    & = D(|\tilde{\eta}_X|^{-1}) D(\tilde{\eta}_X) \tilde{\eta}_{D(X)} |\tilde{\eta}_{D(X)}|^{-1} \\
    & = |\tilde{\eta}_{D(X)}|^{-1} |\tilde{\eta}_{D(X)}|^2 |\tilde{\eta}_{D(X)}|^{-1}
    = \id_{D(X)}. \qedhere
  \end{align*}
\end{proof}

\subsection{Criterion for existence of a lift}
\label{subsec:D0-lift-criterion}

Throughout this subsection, we suppose that $\{ X_i \}_{i \in I}$ is a family of finite-dimensional irreducible $*$-representations of $A$ such that for all $i \in I$, the dual $X_i^\vee$ is isomorphic to $X_j$ for some $j \in I$. The aim of this subsection is to show that the functor $D_0$ can be lifted on the full subcategory
\begin{equation}
  \label{eq:full-sub-gen-by-Xi}
  \mathcal{C} := \{ X \in \Rep(A) \mid
  X \cong X_{i_1} \oplus \dotsb \oplus X_{i_m} \text{ for some $i_1, \dotsc, i_m \in I$} \}.
\end{equation}
We may assume that $X_i \not \cong X_j$ whenever $i \ne j$. Then, for $i \in I$, we can define $i^\vee \in I$ so that $X_i^\vee \cong X_{i^\vee}$. One can define a contravariant endofunctor $D$ on $\mathcal{C}$ by extending $X_i \mapsto X_{i^\vee}$. However, it is difficult to check that the functor $D$ so-defined is a $*$-functor since the inner product on $D(X)$ is not given explicitly. Our approach is first to construct a positive operator $g_X: X \to X$ for each $X \in \mathcal{C}$ and then define $D(X)$ to be the left $A$-module $X^\vee$ with the inner product
\begin{equation}
  \label{eq:D0-lift-ip}
  \langle \phi_x | \phi_y \rangle_{D(X)} := \langle y | g_X(x) \rangle_X \quad (x, y \in X).
\end{equation}

Now we explain how to construct $g_X$. First, let $i \in I$. By the assumption, $X_i^\vee$ has an inner product making it into a $*$-representation. Fix such an inner product $\langle -|- \rangle_0$ and then define $\tilde{g}_i: X_i \to X_i$ by
\begin{equation*}
  \langle y | \tilde{g}_i(x) \rangle = \langle \phi_x | \phi_y \rangle_0 \quad (x, y \in X_i).
\end{equation*}
Since $(y, x) \mapsto \langle \phi_x | \phi_y \rangle_0$ is an inner product on $X_i$, $\tilde{g}_i$ is positive and invertible. Moreover, $\tilde{g}_i(a x) = S^2(a) \tilde{g}_i(x)$ holds for all $a \in A$ and $x \in X$. Now set
\begin{equation*}
  g_i = \Trace(\tilde{g}_i^{-1})^{1/2} \Trace(\tilde{g}_i)^{-1/2} \cdot \tilde{g}_i.
\end{equation*}
Then $g_i$ is an invertible linear map such that
\begin{equation}
  \label{eq:D0-lift-g-0}
  g_i \ge 0, \quad \Trace(g_i) = \Trace(g_i^{-1}) > 0,
  \quad g_i(a x) = S^2(a) g_i(x) \quad (a \in A, x \in X_i).
\end{equation}
Now let $X \in \mathcal{C}$. By the definition of $\mathcal{C}$, there is a canonical isomorphism
\begin{equation}
  \label{eq:D0-lift-can-iso}
  \bigoplus_{i \in I} X_i \otimes_{\mathbb{C}} \Hom_A(X_i, X) \to X,
  \quad (x_i \otimes f_i)_{i \in I} \mapsto \sum_{i \in I} f_i(x_i)
\end{equation}
of left $A$-modules. Recall that $\End_A(X_i) = \mathbb{C} \cdot \id_{X_i}$. One can check that \eqref{eq:D0-lift-can-iso} is unitary if we define an inner product on $\Hom_A(X, X_i)$ by
\begin{equation*}
  \langle f | g \rangle \cdot \id_{X_i} = f^* g \quad (f, g \in \Hom_A(X_i, X)).
\end{equation*}
By using \eqref{eq:D0-lift-can-iso}, we define $g_X: X \to X$ so that the diagram
\begin{equation*}
  \begin{CD}
    X @>>> \bigoplus_{i \in I} X_i \otimes_{\mathbb{C}} \Hom_A(X_i, X) \\
    @V{g_X}V{}V @V{}V{(g_i \otimes \id)_{i \in I}}V \\
    X @>>> \bigoplus_{i \in I} X_i \otimes_{\mathbb{C}} \Hom_A(X_i, X)
  \end{CD}
\end{equation*}
commutes. By \eqref{eq:D0-lift-g-0} and the unitarity of \eqref{eq:D0-lift-can-iso}, we obtain:
\begin{gather}
  \label{eq:D0-lift-g-1}
  g_X \ge 0,
  \quad \Trace(g_X^{-1}) = \Trace(g_X), \\
  \label{eq:D0-lift-g-3}
  g_X(a x) = S^2(a) g_X(x)
\end{gather}
for all $a \in A$ and $x \in X \in \mathcal{C}$. If $f: X \to Y$ is a morphism in $\mathcal{C}$, then
\begin{equation}
  \label{eq:D0-lift-g-4}
  g_Y \circ f = f \circ g_X
\end{equation}
by the definition of $g_X$'s.

To summarize results so far, we introduce the following notation: Given a ring homomorphism $\alpha: R_1 \to R_2$, we denote by $\alpha^\natural: \Mod{R_2} \to \Mod{R_1}$ the pull-back functor along $\alpha$. Since $S^2: A \to A$ is an algebra automorphism, it induces a functor
\begin{equation*}
  S^{2\natural}: \Mod{A} \to \Mod{A}.
\end{equation*}
Let $U: \mathcal{C} \to \Mod{A}$ be the forgetful functor. Then:

\begin{lemma}
  \label{lem:D0-lift-g}
  The family $g = \{ g_X \}_{X \in \mathcal{C}}$ defines a natural isomorphism
  \begin{equation*}
    g: U \to S^{2\natural} U
  \end{equation*}
  satisfying~\eqref{eq:D0-lift-g-1}. Such a natural isomorphism is unique.
\end{lemma}
\begin{proof}
  The first sentence is nothing more than a paraphrase of \eqref{eq:D0-lift-g-1}-\eqref{eq:D0-lift-g-4}. To show the uniqueness, let $g': U \to S^{2\natural} U$ be another such natural isomorphism. Set $g_i = g_{X_i}$ and $g'_i = g'_{X_i}$. It is sufficient to show that $g_i = g'_i$ for all $i \in I$. Now let $i \in I$. Since $X_i$ and $S^{2\natural}(X_i)$ are finite-dimensional simple $A$-modules, $g'_i = c g_i$ for some $c \in \mathbb{C}^\times$ by Schur's lemma. By \eqref{eq:D0-lift-g-1},
  \begin{equation*}
    c \Trace(g_i) = \Trace(g'_i) = \Trace(g_i'^{-1}) = c^{-1} \Trace(g_i) = c^{-1} \Trace(g_i).
  \end{equation*}
  By~\eqref{eq:D0-lift-g-0}, $c = \pm 1$. Since $g_i$ and $g'_i$ are positive, $c$ must be $+1$.
\end{proof}

The natural isomorphism $g$ plays the role of the element $g$ in \S\ref{subsec:RF-st-alg}. For $X \in \mathcal{C}$, we define $D(X) \in \mathcal{C}$ to be the left $A$-module $X^\vee$ endowed with the inner product given by \eqref{eq:D0-lift-ip}. The following lemma is proved in a similar way to \S\ref{subsec:RF-st-alg}:

\begin{lemma}
  $X \mapsto D(X)$ extends to a lift of $D_0$ on $\mathcal{C}$.
\end{lemma}

Hence, by Proposition~\ref{prop:D0-lift}, there exists a natural isomorphism $\eta: \id_{\mathcal{C}} \to D D$ such that the pair $(D, \eta)$ is a $*$-compatible dual structure for $\mathcal{C}$. In our situation, we can construct $\eta$ by using $g$ of Lemma \ref{lem:D0-lift-g}. We need to prove:

\begin{lemma}
  \label{lem:D0-lift-g-dual}
  Let $g$ be the natural isomorphism of Lemma \ref{lem:D0-lift-g}. Then
  \begin{equation*}
    g_{D(X)} \circ (g_X)^{\vee} = \id_{D(X)}
    \quad (X \in \mathcal{C}).
  \end{equation*}
\end{lemma}
\begin{proof}
  Set $\alpha_X = (g_X)^\vee$ and $\beta_X = (g_{D(X)})^{-1}$ for $X \in \mathcal{C}$. Since both $\alpha = \{ \alpha_X \}_{X \in \mathcal{C}}$ and $\beta = \{ \beta_X \}_{X \in \mathcal{C}}$ are natural isomorphisms $S^{2\natural} U D \to U D$, it is sufficient to show that $\alpha_{X_i} = \beta_{X_i}$ holds for all $i \in I$. Let $i \in I$. By Schur's lemma, $\beta_{X_i} = c \cdot \alpha_{X_i}$ for some $c \in \mathbb{C}^{\times}$. By \eqref{eq:D0-lift-g-1}, we have
  \begin{gather*}
    \Trace(\alpha_{X_i}) = \Trace(g_{X_i}) = \Trace(g_{X_i}^{-1}) = \Trace(\alpha_{X_i}^{-1}).
  \end{gather*}
  We also obtain $\Trace(\beta_{X_i}) = \Trace(\beta_{X_i}^{-1})$ in a similar way. Hence,
  \begin{equation*}
    c \Trace(\alpha_{X_i}) = \Trace(c \alpha_{X_i}) = \Trace(\beta_{X_i}) = \Trace(\beta_{X_i}^{-1})
    = \Trace(c^{-1} \alpha_{X_i}^{-1}) = c^{-1} \Trace(\alpha_{X_i}).
  \end{equation*}
  Now we conclude $c = 1$ by the same way as the proof of Lemma~\ref{lem:D0-lift-g}.
\end{proof}

Define $\eta: \id_{\mathcal{C}} \to D D$ by the same formula as \eqref{eq:dual-module-1} but by using the natural isomorphism $g$ instead of $g \in A$. We also define an anti-linear functor $J: \mathcal{C} \to \mathcal{C}$ and a natural isomorphism $i: \id_{\mathcal{C}} \to J J$ in a similar manner. Then, again in the same way as \S\ref{subsec:RF-st-alg}, we prove:

\begin{proposition}
  \label{prop:D0-lift-equiv}
  The pair $\mathcal{D} = (D, \eta)$ is a $*$-compatible dual structure for $\mathcal{C}$ and the pair $\mathcal{J} = (J, i)$ is a $*$-compatible real structure for $\mathcal{C}$. The dual structures $\mathbb{D}(\mathcal{J})$ and $\mathcal{D}$ are unitary equivalent via the natural isomorphism
  \begin{equation*}
    \varphi_X: J(X) \to D(X), \quad \overline{x} \mapsto \phi_x \quad (x \in X \in \mathcal{C}).
  \end{equation*}
\end{proposition}

\subsection{Proofs of some propositions}

\begin{proof}[Proof of Proposition~\ref{prop:BNS}]
  Let $\{ X_i \}_{i \in I}$ be the complete set of representatives of simple left $A$-modules. Since $A$ is assumed to be a finite-dimensional $C^*$-algebra, we may assume that all $X_i$'s are $*$-representations. Applying the arguments of \S\ref{subsec:D0-lift-criterion} to $\{ X_i \}$, we obtain a natural isomorphism $g$ as in Lemma~\ref{lem:D0-lift-g}. Since $A$ is finite-dimensional, there uniquely exists an element $g \in A$ such that
\begin{equation*}
  g_X(x) = g \cdot x \quad (x \in X \in \fdRep(A))
\end{equation*}
Interpreting~\eqref{eq:D0-lift-g-1}-\eqref{eq:D0-lift-g-3} in terms of the element $g$, we see that the element $g$ has the desired properties. Uniqueness of such an element follows from Lemma~\ref{lem:D0-lift-g}. $S(g) = g^{-1}$ follows from Lemma~\ref{lem:D0-lift-g-dual}.
\end{proof}

\begin{proof}[Proof of Proposition~\ref{prop:BNS-cog}]
  Let $\{ V_i \}_{i \in I}$ be the complete set of representatives of simple $C$-comodules. By Proposition~\ref{prop:compact-coalg}, we may assume that all $V_i$'s are $*$-\hspace{0pt}corepresentations. Now let $A = C^*$ be the dual algebra of $C$ and define $S: A \to A$ by~\eqref{eq:RF-dual-st-alg}. Applying the arguments of \S\ref{subsec:D0-lift-criterion} to $\{ X_i \}$, where $X_i = \Phi(V_i) \in \fdRep(A)$, we obtain a natural isomorphism $g$ as in Lemma~\ref{lem:D0-lift-g}.

  Since $C = \bigoplus_{i \in I} C(V_i)$ (see the proof of Proposition~\ref{prop:compact-coalg}), we can define uniquely a linear map $\gamma: C \to \mathbb{C}$ by the following condition:
  \begin{equation*}
    \langle \gamma, v_{(1)} \rangle v_{(0)} = g_i(v) \quad (i \in I, v \in V_i).
  \end{equation*}
  Now the claim is proved by interpreting the properties of the natural isomorphism $g$ in a similar way as Proposition~\ref{prop:BNS}.
\end{proof}


\end{document}